\documentclass[a4paper,11pt]{amsart}
\RequirePackage{etex}

\usepackage{amssymb}
\usepackage{amsmath}
\usepackage{amsthm}
\usepackage{verbatim}
\usepackage{graphicx}
\usepackage[all]{xy}
\usepackage{url}
\usepackage{tabularx,ragged2e,booktabs,caption}

\usepackage{comment}
\usepackage{url}

\usepackage{longtable}
\usepackage{lscape}

\usepackage[dvipsnames]{xcolor}

\usepackage[T2A,T1]{fontenc}
\usepackage[utf8]{inputenc}
\usepackage[russian,english]{babel}

\usepackage{dynkin-diagrams}
\tikzset{big arrow/.style={
    -Stealth,line cap=round,line width=1mm,
    shorten <=1mm,shorten >=1mm}}

\newcommand{\SmallMatrix}[1]
{\left(\begin{smallmatrix}#1\end{smallmatrix}\right)}

\newcommand{\0}[2]{{}_{#2}#1}

\newcommand{\ii}{\boldsymbol{i}}

\DeclareMathOperator{\Aut}{Aut}

\newcommand{\ad}{{\text{\upshape{ad}}}}

\newcommand{\Zl}[1]{\mathrm{Z}^{#1}}
\newcommand{\Bd}[1]{\mathrm{B}^{#1}}
\newcommand{\Ho}[1]{\mathrm{H}^{#1}}

\numberwithin{equation}{section}

\theoremstyle{plain}

\newtheorem{theorem}[equation]{Theorem}
\newtheorem{proposition}[equation]{Proposition}

\newtheorem{lemma}[equation]{Lemma}
\newtheorem{corollary}[equation]{Corollary}

\newtheorem{main-theorem}[equation]{Main Theorem}

\theoremstyle{definition}

\newtheorem{remark}[equation]{Remark}

\newtheorem{definition}[equation]{Definition}
\newtheorem{notation}[equation]{Notation}
\newtheorem{construction}[equation]{Construction}

\newtheorem{subsec}[equation]{}

\newcommand{\sN}{{\mathcal{N}}}

\newcommand{\into}{\hookrightarrow}
\newcommand{\isoto}{\stackrel{\sim}{\to}}

\newcommand{\labelto}[1]{\xrightarrow{\makebox[1.5em]{\scriptsize ${#1}$}}}
\newcommand{\longisoto}{\labelto\sim}

\newcommand{\hs}{\kern 0.8pt}
\newcommand{\hssh}{\kern 1.2pt}
\newcommand{\hshs}{\kern 1.6pt}
\newcommand{\hssss}{\kern 2.0pt}

\newcommand{\hm}{\kern -0.8pt}
\newcommand{\hmm}{\kern -1.2pt}

\DeclareMathOperator{\Lie}{Lie}

\newcommand{\ttau}{\tau}

\newcommand{\at}{\mathbin{\bar*_\theta}}
\newcommand{\ath}{\mathbin{*_\theta}}
\newcommand{\nb}{{\bar n}}
\newcommand{\tb}{{\bar t}}
\newcommand{\thG}{{\0\GG\theta}}
\newcommand{\thT}{{\0\TT\theta}}

\newcommand{\R}{{\mathbb{R}}}
\newcommand{\Z}{{\mathbb{Z}}}
\renewcommand{\C}{{\mathbb{C}}}
\newcommand{\Q}{{\mathbb{Q}}}

\newcommand{\id}{{\rm id}}

\newcommand{\GG}{{\bf G}}

\newcommand{\NN}{{\bf N}}

\newcommand{\TT}{{\bf T}}
\newcommand{\WW}{{\bf W}}

\newcommand{\X}{{\sf X}}

\newcommand{\tthat}{{\hat\theta}}

\DeclareMathOperator{\Hom}{Hom}

\newcommand{\pp}{p}
\newcommand{\qq}{q}

\newcommand{\GL}{{\bf{GL}}}
\newcommand{\SL}{{\bf{SL}}}

\newcommand{\SO}{{\bf{SO}}}

\newcommand{\PGL}{{\bf{PGL}}}

\DeclareMathOperator{\Inn}{Inn}

\DeclareMathOperator{\Gal}{Gal}
\def\Hquat{{\mathbb{H}}}

\newcommand\two[2]{\underset{\ds#2}{#1}}

\DeclareMathOperator{\Ad}{Ad}

\def\half{{\tfrac{1}{2}}}
\def\ihalf{{\tfrac{\ii}{2}}}

\newcommand{\hh}{\kern 1.0pt}
\newcommand{\ssc}{{\text{\upshape{sc}}}}
\def\Dtil{{\widetilde{D}}}

\def\Smtil{{\widetilde{\Sm}}}
\newcommand{\vvk}{{(k)}}

\newcommand{\ds}{\displaystyle}

\newcommand{\gf}[2]{\genfrac{}{}{0pt}{}{#1}{#2}}

\def\even{{\text{even}}}
\def\odd{{\text{odd}}}

\newcommand{\NNt}{\NN_\tau}
\newcommand{\Nt}{N_\tau}

\DeclareMathOperator{\inn}{inn}
\newcommand{\Ghat}{{\widehat G}}

\newcommand{\AAA}{{\sf A}}
\newcommand{\BBB}{{\sf B}}
\newcommand{\CCC}{{\sf C}}
\newcommand{\DDD}{{\sf D}}
\newcommand{\EEE}{{\sf E}}
\newcommand{\FFF}{{\sf F}}
\newcommand{\GGG}{{\sf G}}

\newcommand{\Wtil}{{\widetilde{W}}}
\newcommand{\Wotil}{{\widetilde{W}_0}}
\newcommand{\BRD}{{\rm BRD}}
\newcommand{\brd}{{\star}}
\newcommand{\Sm}{S}
\newcommand{\Wohat}{\widehat{W}_0}
\newcommand{\Om}{\mathcal{O}}
\newcommand{\Qm}{\mathcal{Q}}
\newcommand{\Rt}{R}
\newcommand{\Rtbar}{{\overline R}}

\newcommand{\wt}{\widetilde}

\newcommand{\tl}{\mathfrak{t}}
\newcommand{\ttl}{\boldsymbol{\tl}}

\newcommand{\hl}{\mathfrak{h}}
\newcommand{\gl}{\mathfrak{g}}
\newcommand{\bl}{\mathfrak{b}}
\newcommand{\zl}{\mathfrak{z}}

\newcommand{\upgam}{{\hs^\gamma\hm}}

\newcommand{\D}{D}

\newcommand{\wtil}{{\widetilde w}}
\newcommand{\what}{{\widehat w}}

\newcommand{\Ga}{\Gamma}

\newcommand{\Km}{{\mathcal K}}
\newcommand{\Zm}{{\mathcal Z}}
\newcommand{\Nm}{{\mathcal N}}
\newcommand{\orb}{{\text{\upshape orb}}}

\newcommand{\eee}{{\text{\tiny\sf E}}}
\newcommand{\aaa}{{\text{\tiny\sf A}}}

\newcommand{\Smbar}{{\overline\Sm}}
\newcommand{\Dbar}{{\overline D}}

\newcommand{\cc}{\raise 1.7pt \hbox{\Tiny{$\bullet$}}}

\newcommand{\htau}{{\hat\tau}}

\newcommand{\sa}{\varphi}
\newcommand{\e}[1]{e_{#1}}
\newcommand{\tto}{t_{\theta}}
\newcommand{\Exp}{{\mathcal{E}}}
\newcommand{\Exphat}{{\widehat{\mathcal E}}}
\newcommand{\tw}{{\mathcal T}}
\newcommand{\der}{{\text{\upshape der}}}
\newcommand{\GmC}{\C^\times}

\newcommand{\lb}{\left\langle}
\newcommand{\rb}{\right\rangle}

\title[Real Galois cohomology]
{Galois cohomology of real semisimple groups\\ via Kac labelings}

\author{Mikhail Borovoi and Dmitry A. Timashev}

\address{Borovoi: Raymond and Beverly Sackler School of Mathematical Sciences,
Tel Aviv University, 6997801 Tel Aviv, Israel}
\email{borovoi@tauex.tau.ac.il}

\address{Timashev:
Lomonosov Moscow State University, Faculty of Mechanics and
Mathematics, Department of Higher Algebra, 119991 Moscow, Russia}
\email{timashev@mech.math.msu.su }

\thanks{Borovoi was partially supported
by the Hermann Minkowski Center for Geometry
and by the Israel Science Foundation (grant 870/16).
Timashev was partially supported by the Russian Foundation
for Basic Research (grant 20-01-00091).}


\dedicatory{To the memory of \`Ernest Borisovich  Vinberg}

\begin{document}


\begin{abstract}
For a connected semisimple group $\GG$ over the field of real numbers $\R$,
using a method of Onishchik and Vinberg,
we compute the first Galois cohomology set $\Ho1(\R,\GG)$
in terms of Kac labelings of the affine Dynkin diagram of $\GG$.
\end{abstract}


\maketitle

\tableofcontents

\setcounter{section}{-1}


\section{Introduction}

\begin{subsec}
Let $\GG$ be a linear algebraic group defined over the field of real numbers $\R$.
For the definition of the  first (nonabelian)
Galois cohomology set $\Ho1(\R,\GG)$ see Serre's book \cite[Section I.5]{Serre};
see also Section \ref{s:anti-regular} below.
Galois cohomology can be used to answer many natural questions
(on classification of real forms of algebraic varieties with additional structure,
on classification of real orbits in a homogeneous space, etc.);
see Serre \cite[Section III.1]{Serre}.
In this article we shall consider the case when $\GG$ is connected and semisimple.
Note that it is very interesting to know the Galois cohomology
of a semisimple or reductive group over a {\em number field,}
but this actually reduces to a calculation of Galois cohomology over $\R$
(and  a calculation of a  certain abelian group);
see Sansuc \cite[Corollary 4.5]{Sansuc}
and  Borovoi \cite[Theorem 5.11]{Borovoi-Memoir}.
\end{subsec}

\begin{subsec}
In this article, by a semisimple or reductive algebraic group we always mean a {\em connected}
semisimple or reductive algebraic group, unless otherwise specified.
A real algebraic group $\GG$ is called {\em $\R$-simple}
if it has no positive-dimensional normal algebraic subgroups defined over $\R$.
A real algebraic group $\GG$ is called {\em absolutely simple}
if it remains simple after extension of scalars from $\R$ to $\C$.

The Galois cohomology sets $\Ho1(\R,\GG)$ of the classical groups are well known.
Recently the sets $\Ho1(\R,\GG)$ were computed for ``most''
of the {\rm absolutely  simple} $\R$-groups  by Adams and Ta{\"\i}bi \cite{AT},
in particular, for all {\em simply connected} absolutely simple $\R$-groups
by Adams and Ta{\"\i}bi \cite{AT} and by Borovoi and Evenor \cite{BE}.
A {\em simply connected} $\R$-simple group $\GG$ that is {\em not} absolutely simple
is isomorphic to a group obtained by the Weil restriction of scalars from
a simply connected simple $\C$-group.
It follows that $\Ho1(\R,\GG)=1$ in this case.
Now if $\GG$ is any {\em simply connected} semisimple $\R$-group,
then $\GG$ is the direct product of its (simply connected) $\R$-simple normal subgroups:
\[\GG=\prod_i\GG_i\hs,\quad\text{and hence}\quad  \Ho1(\R, \GG)=\prod_i\Ho1(\R,\GG_i),\]
where we know $\Ho1(\R,\GG_i)$ for each
$\GG_i$\hs, see above.
Thus  $\Ho1(\R, \GG)$ is known for all {\em simply connected semisimple} $\R$-groups $\GG$.

Kac \cite{Kac} used infinite-dimensional Lie algebras
to classify the conjugacy classes of automorphisms of finite order
of a simple Lie algebra over the field of complex numbers $\C$
in terms of what is now called {\em Kac diagrams.}
For another method giving the same description of the conjugacy classes
of automorphisms of finite order  in terms of Kac diagrams,
see Onishchik and Vinberg \cite[Section 4.4]{OV}, and also
Gorbatsevich, Onishchik, and Vinberg \cite[Section 3.3]{OV2};
the authors write that their approach goes back to Gantmacher \cite{Gant}.
In particular, these methods and results  give a classification of involutions;
see \cite[Section 5.1.5]{OV} and \cite[Section 4.1.4]{OV2}.
A slight modification of these methods and results gives $\Ho1(\R,\GG)$
for all {\em absolutely simple, adjoint} $\R$-groups $\GG$.
Arguing as above, we see that   $\Ho1(\R, \GG)$ is actually known for all {\em adjoint semisimple} $\R$-groups $\GG$.

In the present article we consider  a general semisimple $\R$-group $\GG$,
not necessarily simply connected or adjoint.
We use the method of Onishchik and Vinberg
to compute $\Ho1(\R,\GG)$ in terms of Kac diagrams,
or, as we say, in terms of {\em Kac labelings} of the affine Dynkin diagram of $\GG$.
Our main result is Theorem \ref{t:main-ss}.
It gives a simple uniform combinatorial  description of the Galois cohomology set  $\Ho1(\R,\GG)$ for any semisimple $\R$-group $\GG$.
This permits one to describe easily {\em additional structures} on $\Ho1(\R,\GG)$, see below.
Using this result, we shall compute the Galois cohomology of  {\em reductive} $\R$-groups  in a forthcoming article.
\end{subsec}

\begin{subsec}
In general, the Galois cohomology set $\Ho1(\R,\GG)$
of a linear algebraic $\R$-group $\GG$
has no natural group structure.
It has only a distinguished point, the {\em neutral element}; see Subsection \ref{ss:H1}.
We see that $\Ho1(\R,\GG)$ is just a
pointed set,
and one is tempted to conclude that it has no other structure.

However, these pointed sets  $\Ho1(\R,\GG)$ {\em together} have an important additional structure:  \emph{functoriality}.
Namely, the correspondence
$\GG\rightsquigarrow  \Ho1(\R,\GG)$
is a functor.
This means that for any homomorphism of algebraic $\R$-groups
$\varphi\colon \GG\to \GG'$
we have the induced morphism of pointed sets
\[\varphi_*\colon  \Ho1(\R,\GG)\to  \Ho1(\R,\GG').\]
For  applications it is important to know $\varphi_*$.

Our description of  $\Ho1(\R,\GG)$ for a semisimple $\R$-group $\GG$ in Main Theorem \ref{t:main-ss} permits one
to compute easily the map $\varphi_*$ in the case when $\varphi$ is a {\em normal} homomorphism,
that is, when $\varphi(\GG)$ is normal in $\GG'$.
See Section~\ref{s:twisting}, where the case of an {\em isogeny} $\GG\to\GG'$ is treated in detail.

Another additional structure on $\Ho1(\R,\GG)$ is \emph{twisting}.
Let $\GG$ be an algebraic $\R$-group and  $a\in \Zl1(\R,\GG)$ be a cocycle.
Consider the inner twist  $_a\GG$ of $\GG$, which is  another $\R$-group; see Subsection \ref{ss:twist-def}.
There is a canonical bijection $\tw_a\colon \Ho1(\R,\hs_a\GG)\to \Ho1(\R,\GG)$
sending the neutral element of $\Ho1(\R,\hs_a\GG)$ to the class of $a$ in $\Ho1(\R,\GG)$.
For a  semisimple $\R$-group $\GG$, using our description of $\Ho1(\R,\GG)$ in Main Theorem \ref{t:main-ss},
we can easily compute the map $\tw_a$; see   Proposition \ref{p:inner-twist}.

For many applications of Galois cohomology
(for instance, for classification of real orbits in real loci of complex homogeneous spaces)
it is important to know explicit cocycles representing each cohomology class in $\Ho1(\R,\GG)$.
Our Main Theorem \ref{t:main-ss} gives such representatives for semisimple $\GG$.
\end{subsec}

\begin{subsec}
We describe the contents of the article.

Sections \ref{s:anti-regular}--\ref{s:prelim-compact} contain preliminary material on
real algebraic groups, Galois cohomology over $\R$, and structure, automorphisms, and real forms of reductive groups.
In Section~\ref{s:red-to-torus} we explain how to reduce the computation of Galois cohomology
of a reductive $\R$-group $\GG$ to a calculation with a maximal compact subtorus $\TT_0$ of~$\GG$.
In Section~\ref{s:M} we identify $\Ho1(\R,\GG)$ with the orbit set for the action
of a certain subgroup of the normalizer of $\TT_0$
on the set $(T_0)_2$ of elements in $\TT_0$ of order $\le2$ by twisted conjugation.
We get rid of the twist in Section~\ref{s:shifting} by shifting $(T_0)_2$
to a disconnected group with identity component $\GG$.

Section~\ref{s:lattices} collects some standard facts about restricted roots of $\TT_0$ and the (twisted) affine Weyl group,
to be used in the next section. In Section~\ref{s:logarithm} we
reduce our problem to a description of the orbit set for a certain discrete group
of affine transformations of the Lie algebra of $\TT_0$.
We obtain a description for this orbit set in Sections \ref{s:inner}--\ref{s:nas}
for simple $\R$-groups and in Section~\ref{s:sqroots} for an arbitrary semisimple $\R$-group.

In Section~\ref{s:mt-ss} we deduce our main result describing $\Ho1(\R,\GG)$ for semisimple $\GG$ in terms of Kac labelings.
Functoriality of our description
is discussed in Section~\ref{s:twisting}.
Sections \ref{s:E7}--\ref{s:E7xSL(m,H)} contain examples of explicit computation of Galois cohomology by our method.
\end{subsec}

\begin{subsec}
{\bf Notation and conventions}\label{not-conv}
\begin{itemize}
\item[\cc] $\Z$ denotes the ring of integers.
\item[\cc] $\Q$, $\R$, and $\C$ denote the fields of rational numbers,
   of real numbers, and of complex numbers, respectively.
\item[\cc] $\ii\in\C$ is such that $\ii^2=-1$. (Our results do not depend on the choice of $\ii$.)
\item[\cc] $\Gamma=\Gal(\C/\R)=\{1,\gamma\}$, the Galois group of $\C$ over $\R$,
           where $\gamma$ is the complex conjugation.
\item[\cc] We denote real algebraic groups  by boldface letters $\GG$, $\TT$, \dots,
           their complexifications by respective Roman (non-bold) letters $G=\GG\times_\R \C$, $T=\TT\times_\R\C$, \dots,
           and the corresponding complex Lie algebras by respective lowercase Gothic letters $\gl=\Lie G$, $\tl=\Lie T$, \dots.
\item[\cc] For a homomorphism $\varphi:G\to H$ of algebraic (or Lie) groups,
           the differential at the unity $d\varphi:\gl\to\hl$ is a homomorphism of Lie algebras.
           By abuse of notation, we often write $\varphi$ instead of $d\varphi$.
\item[\cc] $G^0$ denotes the identity component of an algebraic (or Lie) group $G$.
\item[\cc] $\Inn G$ denotes the group of inner automorphisms of a group $G$.
\item[\cc] $\inn(g)\colon x\mapsto gxg^{-1}$ denotes the inner automorphism
of a group $G$ corresponding to an element $g$ of $G$.
\item[\cc] $\GG$ is a connected reductive $\R$-group, unless otherwise specified.
\item[\cc]  $Z(\GG)$ denotes the center of $\GG$.
\item[\cc] $\Zm_\GG(\ )$ denotes the centralizer in $\GG$ of the set in the parentheses.
\item[\cc] $\Nm_\GG(\ )$ denotes the normalizer in $\GG$ of the group in the parentheses.
\item[\cc] $\GG^\ad:=\GG/Z(\GG)$  denotes the corresponding adjoint group.
\item[\cc] $\GG^\der=[\GG,\GG]$, the derived group of $\GG$.
\item[\cc] $\GG^\ssc$ denotes the universal cover of the connected semisimple group $\GG^\der$.
\item[\cc] $\TT\subset \GG$ is a maximal torus (defined over $\R$).
\item[\cc] $\TT^\ad:=\TT/Z(\GG)$ denotes the image of $\TT$ in $\GG^\ad$, which is a  maximal torus in $\GG^\ad$.
\item[\cc] $\TT^\ssc$ denotes the preimage of $\TT$ in $\GG^\ssc$, which is a maximal torus in $\GG^\ssc$.
\item[\cc] $\X^*(T)=\Hom(T,\GmC)$, the character group of $T$, where $\Hom$
          denotes the group of homomorphisms of {\em algebraic $\C$-groups}.
          We regard $\X^*(T)$ as a lattice in the dual space $\tl^*$ of $\tl$,
          in view of the canonical embedding ${\X^*(T)\into\tl^*}$, $\chi\mapsto d\chi$.
\item[\cc] $\X_*(T)=\Hom(\GmC,T)$, the cocharacter group of $T$.
          We regard $\X_*(T)$ as a lattice in $\tl$, in view of the canonical embedding $\X_*(T)\into\tl$, $\nu\mapsto d\nu(1)$.
\item[\cc] $A_2$ denotes the set of elements of order dividing $2$ in a subset $A$ of some group.
\item[\cc] $A_2^a$ denotes the set of elements of $A$ with square $a$.
\item[\cc] By an involution (of an algebraic group, based root datum etc.) we mean
    an automorphism with square identity. In particular, we regard the identity automorphism as the trivial involution.
\end{itemize}
\end{subsec}

\noindent{\bf Acknowledgements.} \
This article was inspired by e-mail correspondence
of the first-named author with E.~B.~Vinberg  in 2008,
where Vinberg proposed a formula for $\Ho1(\R,\GG)$ in terms of Kac labelings of the extended Dynkin diagram of $\GG$
in the case when $\GG$ is an inner form of a simply connected compact $\R$-group; see Corollary \ref{c:main-sc}.
The authors are grateful to Benjamin McKay for developing a new version of his package {\tt dynkin-diagrams}
especially for this article and to Zinovy Reichstein for helpful comments.
The first-named author worked on this article, in particular, during his visit
to the Institut des Hautes \'Etudes Scientifiques (IHES) in the fall of  2019,
and he is grateful to this institute for support and excellent working conditions.

\section{Real algebraic groups and Galois cohomology}
\label{s:anti-regular}

\begin{subsec}
Let $\GG$ be a real linear algebraic group.
In the coordinate language, the reader may regard $\GG$ as
a subgroup in the general linear group $\GL_n(\C)$ (for some integer~$n$)
defined by polynomial equations with {\em real} coefficients in the matrix entries;
see Borel \cite[Section 1.1]{Borel-66}.
More conceptually, the reader may assume that $\GG$ is an affine group scheme
of finite type over $\R$; see  Milne \cite[Definition 1.1]{Milne}.
With any of these two equivalent  definitions, $\GG$ defines a covariant functor
\begin{equation*}
A\mapsto \GG(A)
\end{equation*}
from the category of commutative unital $\R$-algebras to the category of groups.
Applying this functor to the $\R$-algebra $\R$, we obtain a real Lie group $\GG(\R)$.
Applying this functor  to the $\R$-algebra $\C$ and to the morphism of $\R$-algebras
\[\gamma\colon \C\to\C, \quad z\mapsto \bar z\quad\text{for }z\in \C,\]
we obtain a complex Lie group $\GG(\C)$ together with an anti-holomorphic involution
$\GG(\C)\to \GG(\C),$
which will be denoted by $\sigma_\GG$.
The Galois group $\Gamma$ naturally acts on $\GG(\C)$; namely, the complex conjugation $\gamma$ acts by $\sigma_\GG$.
We have $\GG(\R)=\GG(\C)^\Gamma$ (the subgroup of fixed points).

We shall consider the linear algebraic group $\GG\times_\R\C$
obtained from $\GG$ by extension of scalars from $\R$ to $\C$.
In this article we shall denote $\GG\times_\R\C$ by $G$,
the same Latin  letter, but non-boldface
(though the standard notation is $\GG_\C$).
By abuse of notation we shall identify $G$ with $\GG(\C)$;
in particular, we shall write $g\in G$ meaning that $g\in \GG(\C)$.

Since $G$ is an affine group scheme over $\C$, we have the ring of regular function
$\C[G]=\R[\GG]\otimes_\R\C.$
Our anti-holomorphic involution $\sigma_\GG$ of $\GG(\C)$ is {\em anti-regular} in the following sense:
when acting on the ring of holomorphic functions on $G$, it preserves the subring $\C[G]$.
An anti-regular involution of $G$ is called also  a {\em real structure on} $G$.
\end{subsec}

\begin{remark}
If $G$ is a {\em reductive} algebraic group over $\C$ (not necessarily connected),
then any anti-holomorphic involution of $G$
is anti-regular. The hypothesis that $G$ is reductive is necessary: the abelian algebraic group $\C\times\C^\times$
has the  anti-holomorphic involution $(z,w)\mapsto (\bar{z},\exp(i\bar{z})\bar{w})$
that is not anti-regular.
See  Adams and Ta\"\i bi \cite[Lemma 3.1]{AT} and Cornulier \cite{Cornulier}.
\end{remark}

\begin{subsec}
A morphism of real linear algebraic groups $\GG\to\GG'$ induces a morphism of pairs
$(G,\sigma_\GG)\to (G',\sigma_{\GG'})$.
In this way we obtain a functor
$\GG\mapsto(G,\sigma_\GG)$.
By Galois descent, see Serre \cite[V.4.20,  Corollary 2 of Proposition 12]{Serre-AGCF} or Jahnel \cite[Theorem 2.2]{Jahnel},
this functor is an equivalence of categories.
In particular, any pair $(G,\sigma)$, where $G$ is a complex linear algebraic group and $\sigma$ is a real structure on $G$,
is isomorphic to a pair coming from a real linear algebraic group $\GG$,
and any morphism of pairs $(G,\sigma)\to (G',\sigma')$ comes
from  a unique  morphism of the corresponding real algebraic groups.

From now on, when mentioning a {\em real} algebraic group $\GG$, we shall actually work with a pair $(G,\sigma)$,
where $G$ is a {\em complex} algebraic group and $\sigma$ is a real structure on $G$.
We shall write $\GG=(G,\sigma)$.
We shall shorten ``real linear algebraic group'' to ``$\R$-group''.
\end{subsec}

\begin{subsec}\label{ss:H1}
Let  $A$ be a $\Ga$-group,
that is, an abstract group (not necessarily abelian) endowed with an action of $\Ga=\Gal(\C/\R)$.
We denote this action by putting a superscript on the left, so that $\gamma$ sends $a\in A$ to $\upgam a$.
We consider the {\em first cohomology pointed set} $\Ho1(\Ga,A)$;
see Serre \cite[Section I.5]{Serre}.
Recall that the group $A$ acts on the set of $1$-{\em cocycles}
\[\Zl1(\Ga, A)=\{z\in A\mid z\cdot\upgam z=1\}\]
by ``twisted conjugation'':
\[a\colon z\mapsto a\cdot z\cdot \upgam a^{-1}\quad\text{for } a\in A,\ z\in \Zl1(\Ga,A).\]
By definition, $\Ho1(\Ga, A)$ is the set of orbits of $A$ in $\Zl1(\Ga,A)$.
In general  $\Ho1(\Ga, A)$  has no natural group structure, but it has a {\em neutral element},
the class of the cocycle $1\in \Zl1(\Ga,A)\subseteq A$.

If the group $A$ is abelian, we consider the  subgroup of  $1$-{\em coboundaries}
\[\Bd1(\Ga, A)=\{a\cdot\upgam a^{-1}\mid a\in A\} \]
of the abelian group $\Zl1(\Ga,A)$.
 Then we have
\[\Ho1(\Ga, A)=\Zl1(\Ga,A)/\Bd1(\Ga,A).\]
Thus $\Ho1(\Ga,A)$ is naturally an abelian group in this case.

Let $\GG=(G,\sigma)$ be a real algebraic group.
The Galois group $\Gamma=\{1,\gamma\}$ acts on $G=\GG(\C)$
(namely, $\gamma$ acts as $\sigma$), and
the \emph{first Galois cohomology set} of $\GG$ is defined as
\[\Ho1(\R,\GG)=\Ho1(\Gamma,G).\]
\end{subsec}

\section{Complex reductive groups}
\label{s:prelim-complex}

\begin{subsec}
Let $G$ be a (connected) reductive group over $\C$, and let $T\subset G$ be a maximal torus.
We denote by
\[ X=\X^*(T)\quad\text{and}\quad X^\vee=\X_*(T) \]
the character and cocharacter lattices of $T$, respectively.
We have a canonical pairing
\begin{equation*}
\langle\,\hs,\hs\rangle: X\times X^\vee\to\Z, \quad
(\chi,\nu)\mapsto \langle\chi,\nu\rangle
\ \text{ for }\chi\in X,\ \nu\in X^\vee,
\end{equation*}
where $\chi\circ \nu=(z\mapsto z^{\langle\chi,\nu\rangle}\hs)$ for $z\in\GmC$.
Under the canonical embeddings $X\into\tl^*$ and $X^\vee\into\tl$ of \ref{not-conv},
this pairing is compatible with the pairing between $\tl^*$ and $\tl$,
and makes the lattices $X$ and $X^\vee$ dual to each other.
\end{subsec}

\begin{subsec}
We consider the root system $\Rt=\Rt(G,T)$.
Then
\[ \gl=\tl\oplus\bigoplus_{\alpha\in \Rt}\gl_{\alpha}\hs,\]
where $\gl_{\alpha}$ is the eigenspace corresponding to the root $\alpha$.
Let $B\subset G$ be a Borel subgroup containing $T$; then
\[ \bl=\tl\oplus\bigoplus_{\alpha\in \Rt_+}\gl_{\alpha}\]
(this formula defines the set of positive roots $\Rt_+\subset \Rt$).
Let $\Sm=\Sm(G,T,B)\subseteq \Rt_+$ denote the corresponding basis of $\Rt$ (system of simple roots).
We denote by $\D=\D(G)=\D(G,T,B)$ the Dynkin diagram of $G$;
then $\Sm$ is identified with the set of vertices of $\D(G,T,B)$.
We do not assume that the Dynkin diagram $\D(G)$ is connected.

The Dynkin diagram $\D(G)$ does not determine $G$ uniquely up to an isomorphism.
We need the notion of the  based root datum of $G$.

Let $G$ be a reductive group over $\C$, $T\subset G$ be a maximal torus,
and $B\subset G$ be a Borel subgroup containing $T$.
Let
\[\BRD(G)=\BRD(G,T,B)=(X,X^\vee,\Rt,\Rt^\vee,\Sm,\Sm^\vee)\]
denote the {\em based root datum of $G$}.
Here $X=\X^*(T)$ is the character group of $T$,
$X^\vee=\X_*(T)$ is the cocharacter group of $T$,
$\Rt=\Rt(G,T)\subset X$ is the root system,
$\Rt^\vee\subset X^\vee$ is the coroot system,
$\Sm=\Sm(G,\,T,\,B)\subset \Rt$ is the basis of $\Rt$ defined by $B$,
$\Sm^\vee\subset \Rt^\vee$ is the corresponding basis of $\Rt^\vee$.
We have a canonical pairing
$X\times X^\vee\to\Z$
and compatible canonical bijections
$\Rt\leftrightarrow \Rt^\vee$, $\Sm\leftrightarrow \Sm^\vee$.
See Springer \cite[Sections 1 and 2]{Springer-Corvallis} for details.
\end{subsec}

\begin{subsec}\label{ss:QP}
Recall the notions of the lattices of roots, weights, coroots and coweights
for an abstract (not necessarily reduced) root system $R$:
the root lattice $Q=Q(R)$ is spanned by all roots in $R$, the coroot lattice $Q^\vee=Q(R^\vee)$
is spanned by all coroots in the dual root system $R^\vee$,
and the weight lattice $P=P(R)$, resp.\ the coweight lattice $P^\vee=P(R^\vee)$,
is the dual lattice of $Q^\vee$, resp.\ of $Q$; see Bourbaki \cite[Section VI.1.9]{Bourbaki} for details.

With any basis $S$ of $R$ (that is, a set of simple roots) there is canonically associated a basis $S^\vee$ of $R^\vee$,
called the set of simple coroots relative to $S$; see \cite[Section VI.1.5, Remark 5]{Bourbaki}.
The sets $S$ and $S^\vee$ are bases of the lattices $Q$ and $Q^\vee$, respectively.
Note that for a non-reduced root system $R$ the set $S^\vee$ is \emph{not} the set of coroots $\alpha^\vee$ over all $\alpha\in S$.
The set of fundamental weights (resp.\ coweights) relative to $S$ is the basis of $P$
(resp.\ of $P^\vee$) dual to $S^\vee$, resp.\ to $S$; see \cite[Section VI.1.10]{Bourbaki}.
(Note that in loc.~cit.\ the fundamental (co)weights are defined
only for a reduced root system, though this restriction is unnecessary.)
\end{subsec}

\begin{subsec}\label{ss:brd}
We recall a construction of the canonical homomorphism
\begin{equation*}
\Aut G\,\to\, \Aut\BRD(G),\quad \sa\mapsto \sa_\brd\,\text{ for }\sa\in \Aut G
\end{equation*}
with kernel $\Inn G$.
Let $\sa\in{\Aut G}$.
Consider the maximal torus $\sa(T)\subseteq G$ and the Borel subgroup $\sa(B)\subseteq G$ containing $\sa(T)$.
Then there exists $g\in G$ such that
\begin{equation}\label{e:phi-T-B}
g\cdot \sa(T)\cdot g^{-1}=T\quad\text{and}\quad g\cdot \sa(B)\cdot g^{-1}=B.
\end{equation}
Moreover, if $g'\in G$ is another such element, then $g'=tg$ for some $t\in T$.
The automorphism $\inn(g)\circ\sa$ of $G$ preserves $T$ and $B$,
and thus induces an automorphism $\sa_\brd$ of $\BRD(G,T,B)$.
Namely,
\begin{align*}
(\sa_\brd\chi)(t) &=\chi(\sa^{-1}(g^{-1}tg))\quad \text{for } \chi\in X,\ t\in T,&\\
(\sa_\brd\nu)(z) &= g\hs\sa(\nu(z))g^{-1}\quad \text{for } \nu\in X^\vee,\ z\in\C^\times.
\end{align*}
One checks immediately that $\sa_\brd$ preserves the pairing $X\times X^\vee\to\Z$.
It follows from \eqref{e:phi-T-B} that $\sa_\brd$ preserves the subsets
$S\subset R\subset X$ and $S^\vee\subset R^\vee\subset X^\vee$.
One can easily check that $\sa_\brd$ does not depend on the choice of $g\in G$
and that the map $\sa\mapsto \sa_\brd$ is a homomorphism.
See, for instance, \cite[Section 3.2 and Proposition 3.1(a)]{BKLR} for details.
Note that there is  a canonical homomorphism
\[\Aut\BRD(G)\to\Aut\D(G),\]
which is injective when $G$ is semisimple.
\end{subsec}

\begin{definition}\label{d:pinning}
A {\em pinning} of $(G,T,B)$ is a family $(\e\alpha)_{\alpha\in\Sm}$, where $\e\alpha\in \gl_\alpha$,
 $\e\alpha\neq 0$ for all $\alpha\in\Sm$.
\end{definition}

\begin{lemma}[see Conrad {\cite[Prop.\,1.5.5]{Conrad-Reductive}}
or Milne {\cite[Prop.\,23.44]{Milne}}]
\label{l:pinned}
Let $(\e\alpha)_{\alpha\in\Sm}$ be a pinning of $(G,T,B)$.
Then the natural homomorphism
\[\Aut(G,T,B,(\e\alpha))\to\Aut\BRD(G,T,B)\]
is an isomorphism.
\end{lemma}

Inverting the isomorphism of Lemma \ref{l:pinned}, we obtain a homomorphism
\begin{equation}\label{e:phi-u}
\Aut\BRD(G,T,B)\isoto\Aut(G,T,B,(\e\alpha))\into\Aut(G),
\end{equation}
which is a splitting of the homomorphism
\begin{equation*}
\label{e:psi-inn}
\Aut G\,\to\, \Aut\BRD(G),\quad \sa\mapsto \sa_\brd.
\end{equation*}
It follows that the latter homomorphism is surjective.
By abuse of notation, we identify $\Aut\BRD(G)$ with its image in $\Aut G$
under the embedding \eqref{e:phi-u} and get an isomorphism
\begin{equation*}
\Aut G \isoto\Inn G\rtimes\Aut\BRD(G)
\end{equation*}
depending on the chosen pinning $(e_\alpha)$.

\section{Compact reductive groups}
\label{s:prelim-compact}

A (connected) reductive $\R$-group $\GG=(G,\sigma)$ is called {\em compact}
if the real Lie group $\GG(\R)=G^\sigma$ is compact.
In this case the real structure $\sigma$ on $G$ is called \emph{compact}.

\begin{proposition}[Weyl, Chevalley]
Any  reductive $\C$-group $G$ admits a compact real structure.
Any two compact real structures on $G$ are conjugate by an inner automorphism of $G$.
\end{proposition}

\begin{proof}
See, for instance, Onishchik and Vinberg \cite[5.2.3, Theorems 8 and 9]{OV}.
\end{proof}

\begin{definition}[Adams and Ta{\"\i}bi {\cite[Def.\,3.12]{AT}}]
Let $\sigma$ be a real structure on a reductive $\C$-group $G$.
A {\em Cartan involution for }$\sigma$ is an involutive regular automorphism $\theta$ of $G$ commuting with $\sigma$
and such that the anti-regular automorphism $\sigma\circ\theta=\theta\circ\sigma$ is a compact real structure on $G$.
\end{definition}

The Cartan involution $\theta$ as defined above is the complexification
of the classical Cartan involution of the real Lie group $G^\sigma$.

\begin{proposition}
\label{p:AT-3-13}
Let $\sigma$ be a real structure on a reductive $\C$-group $G$.
Then there exists a Cartan involution $\theta$ for $\sigma$.
The correspondence $\sigma\rightsquigarrow \theta$ induces a bijection
between the set of conjugacy classes of real structures on $G$
and the set of conjugacy classes of involutive regular automorphisms of $G$,
where in both cases
the conjugation action of inner automorphisms of $G$ is considered.
\end{proposition}

This result is well known for semisimple $G$; see for instance Onishchik and Vinberg \cite[5.1.4, Theorems 3 and 4]{OV}.
For a proof for a general reductive group (even not necessarily connected)
see Adams and Ta{\"\i}bi \cite[Theorem 3.13(1)(a) and Corollary 3.17]{AT}.

\begin{subsec}
By Proposition \ref{p:AT-3-13}  any reductive $\R$-group is obtained
from a compact reductive $\R$-group $\GG=(G,\sigma_c)$
by twisting by an involutive real automorphism $\theta$ of $\GG$.
In other words, any reductive $\R$-group is of the form
$$\thG=(G,\sigma),$$ where $\sigma=\theta\circ\sigma_c$, $\sigma_c$
is a compact structure on $G$, $\theta\in\Aut\GG$, $\theta^2=\id$.
Our aim is to compute $\Ho1(\R,\thG)$. To this end, we need to analyze the structure of $\theta$.

The fixed point subgroup $\GG^{\theta}$ is {reductive (though possibly disconnected)},
 and $(\GG^\theta)^0$ is compact.
Choose a maximal torus $\TT_0\subset(\GG^{\theta})^0$.
\end{subsec}

\begin{lemma}\label{l:Z(T0)}
$\TT=\Zm_\GG(\TT_0)$ is a $\theta$-stable maximal torus in $\GG$.
\end{lemma}

This is a particular case of a more general statement about any semisimple automorphism $\theta$ of $G$;
see for instance \cite[3.3.8, Theorem 3.13]{OV2}.
For convenience of the reader, we provide a simple proof adapted to our case.

\begin{proof}
The group $\TT$ is connected and reductive; see Humphreys \cite[Theorem 22.3 and Corollary 26.2.A]{Humphreys}.
Since $\TT$ is $\theta$-stable, $\tl$ is a graded Lie subalgebra of $\gl$,
where the grading $\gl=\gl^{\theta}\oplus\gl^{-\theta}$ modulo 2 is given by the decomposition
into the sum of eigenspaces $\gl^{\pm\theta}$ for the differential of $\theta$ with eigenvalues $\pm1$, respectively.

Since $\tl_0$ is a Cartan subalgebra in $\gl^{\theta}$, it coincides with its centralizer in $\gl^{\theta}$.
Hence we have $\tl\cap\gl^{\theta}=\tl_0$.
The reductive Lie algebra $\tl$ is the direct sum of its center $\zl(\tl)$ and the derived subalgebra  $\tl^\der=[\tl,\tl]$.
Both summands are graded and, since $\zl(\tl)\supseteq\tl_0=\tl\cap\gl^{\theta}$, we have $\tl^\der\subseteq\gl^{-\theta}$.
It follows  that $[\tl^\der,\tl^\der]\subseteq \gl^\theta$.
Since also $[\tl^\der,\tl^\der]\subseteq\tl^\der\subseteq\gl^{-\theta}$,
we see that $[\tl^\der,\tl^\der]=0$.
We conclude that  $\tl$ is abelian, because otherwise $\tl^\der$ would be nonzero and semisimple, hence nonabelian.
Therefore $\TT$ is a torus, and clearly it is a maximal torus containing $\TT_0$.
\end{proof}

\begin{subsec}
We have a decomposition $\TT=\TT_0\cdot\TT_1$ into an almost direct product of algebraic tori,
where $\theta$ acts on $\TT_0$ trivially and on $\TT_1$ as inversion.
The respective Lie algebra decomposition is $\tl=\tl_0\oplus\tl_1$,
where $\tl_0=\tl\cap\gl^{\theta}$ and $\tl_1=\tl\cap\gl^{-\theta}$.
Note that $\thT$ is a maximal torus in $\thG$,
$\thT_1$ is the maximal split subtorus in $\thT$,
and $\thT_0=\TT_0$ is the maximal compact subtorus in $\thT$.
Since $\thT$ is the centralizer of $\TT_0$ in $\thG$, we see  that $\TT_0$ is a maximal compact torus in $\thG$.

Consider the root system $\Rt=\Rt(G,T)$, on which $\theta$ acts naturally.
The restrictions of all roots $\alpha\in\Rt$ to $T_0$ are nonzero (in the additive terminology),
because $\Zm_G(T_0)=T$. Choose a sufficiently general lattice vector $\nu\in(X^{\vee})^{\theta}=\X_*(T_0)$,
so that $\langle\alpha,\nu\rangle\ne0$, $\forall\alpha\in\Rt$,
and define the set of positive roots $\Rt_+\subset\Rt$ by the condition $\langle\alpha,\nu\rangle>0$.
Then $\Rt_+$ and the corresponding set of simple roots $\Sm\subset\Rt_+$ are preserved under $\theta$.
We denote by $B$ the Borel subgroup of $G$ containing $T$ which corresponds to $\Rt_+$.
\end{subsec}

\begin{proposition}
\label{p:inv-can-form}
Let $\GG=(G,\sigma_c)$ be a compact reductive $\R$-group, $\theta$ be an involutive real automorphism of $\GG$,
and $\TT=\Zm_\GG(\TT_0)$ be the centralizer of a maximal torus $\TT_0\subset\GG^{\theta}$.
Set $\tau=\theta_\brd\in\Aut\BRD(G,T,B)$. There exists a pinning $(\e\alpha)_{\alpha\in\Sm}$ such that
\begin{equation*}
\label{e:inv-can-form}
\theta=\inn(t_0)\circ\ttau,
\end{equation*}
where $\ttau$ is regarded as an involutive automorphism of $(G,T,B,(\e\alpha))$
via the isomorphism \eqref{e:phi-u} and $t_0\in T_0=(T^{\ttau})^0$ is such that $t_0^2\in Z(G^\der)$.
\end{proposition}

Again, this is a particular case, for which we provide a simple proof,
of a more general statement about any semisimple automorphism $\theta$ of $G$;
see for instance \cite[3.3.8, Theorems 3.12(4) and 3.13]{OV2}.

\begin{proof}
First, take an arbitrary pinning $(\e\alpha)$ of $(G,T,B)$ and consider
the respective automorphism $\ttau$ of $(G,T,B,(\e\alpha))$.
Then $\theta$ and $\ttau$ act similarly on $T$, whence $\theta=\inn(t)\circ\ttau$ for some $t\in\Zm_G(T)=T$.

We decompose: $t=t_0t_1$, where $t_0\in T_0$ and $t_1\in T_1$. Choose $s\in T_1$ such that $s^2=t_1$.
Easy calculations show that the
automorphism $\inn(t_1)\circ\ttau$
is still involutive and preserves the pinning $(\Ad(s)\hshs\e\alpha)$.
Replacing $\ttau$ by $\inn(t_1)\circ\ttau$,  we obtain $\theta=\inn(t_0)\circ\ttau$, as desired.

Finally, note that $\TT_0$ decomposes into a product of maximal compact tori in $\thG^\der$ and in $Z(\thG)^0$.
Decomposing $t_0$ respectively, we see that only the first factor contributes to $\inn(t_0)$.
Thus we may assume $t_0\in T_0\cap{G^\der}$. Since $\ttau$ commutes with $\inn(t_0)$,
and both $\ttau$ and $\theta$ are involutive, $\inn(t_0)$ is also involutive, whence $t_0^2\in Z(G^\der)$.
\end{proof}

\begin{remark}\label{r:tau&comp}
Since $t_0$ is an element of finite order, it belongs to $\TT_0(\R)$.
Hence $\inn(t_0)$ and $\ttau$ commute with $\sigma_c$, that is, belong to $\Aut\GG$.
\end{remark}

\section{Reduction to a maximal compact torus}
\label{s:red-to-torus}

From now on we shall use the following notations:

\begin{subsec}{\bf Notation}\label{n:not2}
\begin{itemize}
\item[\cc] $\GG=(G,\sigma_c)$ is a {\em compact} connected reductive $\R$-group.
           We denote the action of $\sigma_c$ on $G$ by bar: $\bar{g}=\sigma_c(g)$ for $g\in G$.
\item[\cc] $\TT\subset\GG$ is a maximal torus.
\item[\cc] $B\subset G$ is a Borel subgroup containing $T$.
\item[\cc] $\tau\in\Aut\BRD(G,T,B)$, $\tau^2=\id$. We also regard $\ttau$
           as a  real automorphism of $(\GG,\TT,B)$ (see Remark~\ref{r:tau&comp})
            preserving a given pinning $(\e\alpha)_{\alpha\in\Sm}$.
\item[\cc] $\TT_0=\left(\TT^\ttau\right)^0$ and $\TT_1=\left(\TT^{-\ttau}\right)^0$,
            where $T^{\pm\ttau}=\{t\in T\mid \ttau(t)=t^{\pm1}\}$. Then $\TT=\TT_0\cdot\TT_1$.
            Note that both $\TT_0$ and $\TT_1$ are {\em compact} tori and $\TT(\R)=\TT_0(\R)\cdot\TT_1(\R)$.
\item[\cc] $\tto\in T_0$ is such that $\tto^2\in Z(G^\der)$.
\item[\cc] $\theta=\inn(\tto)\circ\ttau$ {is an involutive automorphism of $\GG$.
           It preserves $\TT$ and $B$, and acts on the Dynkin diagram $\D(G)$ by the involution $\tau$.}
\item[\cc] $\thG=(G,\sigma)$, where $\sigma=\theta\circ\sigma_c$.
\item[\cc] $\NN=\Nm_\GG(\TT)$ {and} $\NN_0=\Nm_\GG(\TT_0)$. Note that $\NN_0\subseteq\NN$, because $\TT=\Zm_\GG(\TT_0)$.
\item[\cc] $W=N/T$ and $W_0=N_0/T$ are the Weyl groups of $T$ and $T_0$ in $G$, respectively.
\end{itemize}
\end{subsec}

\begin{subsec}
We wish to compute $\Ho1(\R,\thG)$.
The set of 1-cocycles for $\thG$ is by definition
$$\Zl1(\R,\thG)=\{a\in G\mid a\cdot\theta(\bar{a})=1\}.$$
The group $\thG(\C)=G$ acts on $\Zl1(\R,\thG)$ by the formula
\begin{equation}\label{e:cohomologous}
g:a\mapsto g\cdot a\cdot\theta(\bar{g})^{-1} \quad\text{for } g\in G,\ a\in \Zl1(\R,\thG),
\end{equation}
and by definition $\Ho1(\R,\thG)$ is the set of orbits for this action.
See Subsection \ref{ss:H1}.

We shall show that the action of $N_0\subset G$ on $T$
 by formula \eqref{e:cohomologous} preserves $\Zl1(\R,\thT)$
and induces an action of $W_0$ on $\Ho1(\R,\thT)$.

Let $n\in N_0$, $s\in T$;
then $nsn^{-1}\in T$.
We define  actions $\ath$ and $\at$ of $N_0$ on $T$ by
\begin{align}
&n\ath s:=n\hs s\,\theta(n)^{-1}=nsn^{-1}\cdot n\,\theta(n)^{-1}, \label{e:ath} \\
&n\at s:=n\hs s\,\theta(\nb)^{-1}=nsn^{-1}\cdot n\,\theta(\nb)^{-1}.\label{e:at}
\end{align}
These actions are well defined by the following lemma.
\end{subsec}

\begin{lemma}\label{l:n-th-s}
If $n\in N_0$, $s\in T$, then $n\ath s\in T$ and $n\at s\in T$.
\end{lemma}

\begin{proof}
Since $\theta$ acts trivially on $T_0$ and  since $W_0$ naturally embeds into $\Aut(T_0)$,
we see that $\theta$ acts trivially on $W_0$.
It follows that  $n\cdot \theta(n)^{-1}\in T$, and hence
\begin{equation*}
 n\ath s=nsn^{-1}\cdot n\, \theta(n)^{-1}\in T.
 \end{equation*}
Since $\TT_0$ is a {\em compact} torus, all its complex automorphisms are defined over $\R$;
hence the complex conjugation acts trivially on $\Aut(T_0)$ and on $W_0$, whence
\[n \nb^{-1}\in T\quad\text{and}\quad \theta(n)\cdot \theta(\nb)^{-1}=\theta(n\nb^{-1})\in T.\]
Consequently,
\begin{equation*}
n\cdot\theta(\nb)^{-1}=n\hs\theta(n)^{-1}\cdot\theta(n)\hs\theta(\nb)^{-1}\in T,
\end{equation*}
thus
\[n\at s=n\hs s\hs\theta(\nb)^{-1}=nsn^{-1}\cdot n\,\theta(\nb)^{-1}\in T.\qedhere\]
\end{proof}

Let $n\in N_0$ and let $s\in \Zl1(\R, \thT)\subset T$,
then by Lemma \ref{l:n-th-s} we have $n\at s\in T$.
Comparing \eqref{e:cohomologous} and \eqref{e:at},
we see that $n\at s$ is cohomologous to $s$ in $\Zl1(\R,\thG)$,
and hence $n\at s$ is a cocycle,
that is, $n\at s\in \Zl1(\R,\thT)$.
We obtain an action of $N_0$ on $\Zl1(\R,\thT)$,
which we again denote by $\at$.

\begin{lemma}\label{l:action-at}
The action $\at$ of $N_0$ on $\Zl1(\R,\thT)$ induces an action of $N_0$ on $\Ho1(\R,\thT)$,
which in turn induces an action of $W_0$ on $\Ho1(\R,\thT)$.
\end{lemma}

\begin{proof}
Let $s,s'\in \Zl1(\R,\thT)$, and assume that $s'\sim s$,
that is, $s'=t\hs s\hshs\theta(\tb)^{-1}$ for some $t\in T$.
Then
\begin{align*}
n\at s'= ns'\theta(\nb)^{-1}&=n\hs t\hs s\hshs\theta(\tb)^{-1}\hs\theta(\nb)^{-1}\\
&=n\hs t\hs n^{-1}\cdot n\hs s \hshs\theta(\nb)^{-1}\cdot
\theta(\nb)\hshs\theta(\tb^{-1})\hs\theta(\nb)^{-1}
=t'\cdot n\at s\cdot\theta({\bar{t'}})^{-1}
\end{align*}
where $t'=ntn^{-1}$.
Thus $n\at s'\sim n\at s$ and we see that indeed  $N_0$ acts on $\Ho1(\R,\thT)$.

If $t\in T\subset N_0$, then
\[t\at s=t\cdot s\cdot\theta(\tb)^{-1}\sim  s.\]
Thus  $T$ acts on $\Ho1(\R,\thT)$ trivially, and
$\at$ indeed induces an action of $W_0$ on $\Ho1(\R,\thT)$.
\end{proof}

We denote the action of $W_0$ on $\Ho1(\R,\thT)$ of Lemma \ref{l:action-at}  again by $\at$.
We denote the space of orbits of this action by $\Ho1(\R,\thT)/W_0$.

\begin{theorem}[Borovoi {\cite[Theorem 1]{Bo}, \cite[Theorem 9]{Borovoi-arXiv}}]
\label{t:Bor}
The embedding of the abelian group $\Zl1(\R,\thT)$ into $\Zl1(\R,\thG)$ induces a bijection
\[\Ho1(\R,\thT)/W_0\isoto \Ho1(\R,\thG).\]
\end{theorem}

\begin{remark}
Write $\WW=\NN/\TT$. It is easy to see that $\hs_\theta\hm \WW(\R)=W_0$
and that the action $\at$ of $W_0$ on $\Ho1(\R,\hs_\theta\hm \TT)$
is the action of $\hs_\theta\hm \WW(\R)$ related to the cohomology exact sequence
\[\cdots\to\hs _\theta\hm \WW(\R)\to \Ho1(\R,\hs_\theta\hm\TT)\to \Ho1(\R,\hs_\theta \NN)\to\cdots\]
coming from the short exact sequence
\[1\to\hs_\theta\hm\TT\to \hs_\theta \NN\to\hs_\theta\hm\WW\to 1.\]
See Serre \cite[I.5.5, before Proposition 39]{Serre}.
\end{remark}

\begin{subsec}\label{ss:H1(R,T)}
Theorem \ref{t:Bor} reduces the computation of $\Ho1(\R,\thG)$ to computing $\Ho1(\R,\thT)$
and describing the orbit set for the above action of $W_0$ on $\Ho1(\R,\thT)$.
We compute $\Ho1(\R,\thT)$ here and postpone describing the orbit set of $W_0$
to Section~\ref{s:M} (see Proposition~\ref{p:T0-Z1-C}).

Let $s\in (T_0)_2$; then clearly $s\in \Zl1(\R,\TT_0)\subseteq \Zl1(\R,\thT)$.
For any  $s\in T_0\cap T_1$, the equalities $\theta(s)=s$ and $\theta(s)=s^{-1}$
imply $s^2=1$; hence $T_0\cap T_1\subseteq(T_0)_2$.
If $s\in T_0\cap T_1$, we have $s\in \Zl1(\R,\thT_1)$\hs,  because $s\in(T_0)_2\subset \Zl1(\R,\thT)$.
Furthermore, then $s$ is a coboundary in $\thT_1$
(because $\thT_1$ is split and by Hilbert's Theorem~90 $\Ho1(\R,\thT_1)=1$);
hence $s$ is a coboundary in $\thT$.
We see that the embedding $(T_0)_2\into\Zl1(\R,\thT)$ induces a homomorphism
\[(T_0)_2\hs/\left(\hs T_0\cap T_1\right)\to \Ho1(\R,\thT).\]
\end{subsec}

\begin{lemma}[{\cite[Lemma 1]{Bo}, \cite[Lemma 3]{Borovoi-arXiv}}]
\label{l:Bor}
The canonical homomorphism of finite abelian groups
$$(T_0)_2/\left(\hs T_0\cap T_1\hs\right)\to \Ho1(\R,\thT)$$
is an isomorphism.
\end{lemma}

\section{The twisted normalizer of a maximal compact torus}
\label{s:M}

\begin{construction}
We define the twisted normalizer of $\TT_0$ to be  the algebraic $\R$-subgroup $\NNt$ of $\NN_0$ such that
\[ \Nt = \{n\in N_0\ |\ n\cdot T_0\cdot \theta(n)^{-1}=T_0\}
= \{n\in N_0\ |\ n\cdot \theta(n)^{-1}\in T_0\}.\]
It depends only on $\ttau\in \Aut(\GG,\TT,B)$  and not on $t_\theta$. Indeed,
\[n\cdot\theta(n)^{-1}=n\cdot \tto\,\ttau(n)^{-1}\tto^{-1}=n\hs\tto\hs n^{-1}\cdot n\,\ttau(n)^{-1}\cdot \tto^{-1}\hs,\]
where the two extreme factors belong to $T_0$ whenever $n\in N_0$.
\end{construction}

\begin{lemma}\label{l:four}\
\begin{enumerate}
\renewcommand{\theenumi}{\roman{enumi}}
\item\label{i:n/th(n)} $n\cdot \theta(n)^{-1}\in (T_0)_2$ for all $n\in \Nt$.
\item\label{i:TM} $T\cap\Nt=\{t_0 t_1\ |\ t_0\in T_0,\ t_1\in T_1,\ t_1^2\in T_0\cap T_1\}$.
\item\label{i:M:T02}
The $\ath$-action of $\Nt$ preserves the three sets $T_0\supset\TT_0(\R)\supset(T_0)_2$.
\end{enumerate}
\end{lemma}

\begin{proof}
\eqref{i:n/th(n)}
Put $t=n\cdot \theta(n)^{-1}\in T_0$.
Since $t\in T_0$, we have $\theta(t)=t$. Thus
\[t=\theta(t)=\theta(n)\cdot n^{-1}=t^{-1},\]
hence $t^2=1$, that is, $t\in (T_0)_2$, as required.

\eqref{i:TM} If $t\in T$,  $t=t_0 t_1$, where $t_0\in T_0$, $t_1\in T_1$,
then $t\cdot\theta(t)^{-1}=t_1^2$\,.
Hence $t\in T\cap \Nt$ if and only if $t_1^2\in T_0\cap T_1$.

\eqref{i:M:T02}
The conjugation action of any $n\in\Nt$ preserves each of the three sets.
The multiplication by $n\cdot\theta(n)^{-1}$ also preserves these sets by \eqref{i:n/th(n)}.
Now it follows from \eqref{e:ath} that the $\ath$-action preserves the three sets too.
\end{proof}

\begin{lemma}\label{l:N-M}\
\begin{enumerate}
\renewcommand{\theenumi}{\roman{enumi}}
\item\label{i:N0} $N_0=\NNt(\R)\cdot T$;
\item\label{i:M} $\Nt=\NNt(\R)\cdot T_0$.
\end{enumerate}
\end{lemma}

\begin{proof}
Since $\Nt \subseteq N_0$ and $N_0/T=W_0$\hs, for \eqref{i:N0}
it suffices to show that any element $w\in W_0$
can be represented by some element $m\in\NNt(\R)$.
Let $w\in W_0$.
Since $\GG$ is compact, $w$ can be represented by some $n\in \NN(\R)$.
Since $w\in W_0$, we have $n\in N_0\cap\NN(\R)= \NN_0(\R)$.
Set $t=n\at1=n\cdot\theta(n)^{-1}\in \GG(\R)$.
Then by  Lemma~\ref{l:n-th-s} we have $t\in \GG(\R)\cap T=\TT(\R)$.
Write $t=t_0\hs t_1$, where $t_i\in \TT_i(\R)$,
and let $s\in \TT_1(\R)$ be such that $s^2=t_1$.
We set $m=s^{-1} n$, then $m\in \NN_0(\R)$ and
\[m\cdot\theta(m)^{-1}=s^{-1}\hs n\hssss\theta(n)^{-1}\hs s^{-1}
         =t_1^{-1}t=t_0\in \TT_0(\R)\subset T_0,\]
hence $m\in \NNt(\R)$.
Clearly, $m$ represents $w$, which proves \eqref{i:N0}.

Now let $n\in\Nt$.
By \eqref{i:N0}, we may write $n=mt$, where $m\in\NNt(\R)$ and $t\in T$.
Then $t\in\Nt\cap T$.
By Lemma \ref{l:four}\eqref{i:TM} we have $t=t_0\hs t_1$,
where $t_0\in T_0$, $t_1\in T_1$, and $t_1^2\in (T_0)_2$.
Then $t_1\in \TT(\R)$ and hence $t_1\in \TT(\R)\cap\Nt\subset\NNt(\R)$.
Thus $n=mt_1\cdot t_0$, where $mt_1\in \NNt(\R)$ and $t_0\in T_0$, which proves \eqref{i:M}.
\end{proof}

By Lemma \ref{l:four}\eqref{i:M:T02}, the group $\Nt$ acts on $(T_0)_2$ via $\ath$\hs,
and we write $(T_0)_2/\Nt$
for the set of orbits of this action.
By Lemma \ref{l:N-M}\eqref{i:M} we have $\Nt=\NNt(\R)\cdot T_0$, and $T_0$ acts on $(T_0)_2$ trivially.
It follows that
\begin{equation*}
(T_0)_2/\Nt=(T_0)_2/\NNt(\R).
\end{equation*}
Note that the action $\ath$ of $\NNt(\R)$ coincides with $\at$.

\begin{lemma}\label{l:cosets-T0-cap-T1}
Each coset of $T_0\cap T_1$ in $(T_0)_2$ is an orbit for the action $\ath$ of $T\cap\Nt$ on $(T_0)_2$.
\end{lemma}

\begin{proof}
Any element  $t\in T\cap\Nt$ decomposes as $t=t_0t_1$ with $t_0\in T_0$,
$t_1\in T_1$, $t_1^2\in T_0\cap T_1$, by Lemma~\ref{l:four}\eqref{i:TM}.
For any $s\in(T_0)_2$ we have
\begin{equation}\label{e:t*s}
t\ath s=t\hs s\hshs\theta(t)^{-1}=st_1^2,
\end{equation}
hence the $(T\cap\Nt)$-orbit of $s$ is contained in the coset $s\cdot(T_0\cap T_1)$.
Conversely, if $s'\in s\cdot(T_0\cap T_1)$, $s'=st'$ with $t'\in  T_0\cap T_1$,
then we may write $t'=t_1^2$ for some $t_1\in T_1$.
By Lemma \ref{l:four}\eqref{i:TM} we have $t_1\in T\cap\Nt$,
and by \eqref{e:t*s}  we have $t_1\ath s=st'=s'$, which proves the lemma.
\end{proof}

\begin{proposition}\label{p:T0-Z1-C}
The inclusion $(T_0)_2\into \Zl1(\R,\thG)$ induces a bijection
\begin{equation*}
(T_0)_2/\Nt\,\isoto\, \Ho1(\R,\thG).
\end{equation*}
\end{proposition}

\begin{proof}
In view of Lemma~\ref{l:action-at} and Theorem~\ref{t:Bor},
it suffices to show that the inclusion $(T_0)_2\into\Zl1(\R,\thT)$
induces a bijection between the orbit sets for the actions $\ath$
of $\Nt$ on $(T_0)_2$ and $\at$ of $N_0$ on $\Zl1(\R,\thT)$.

We know from Lemma~\ref{l:Bor} that any orbit of $T\subset N_0$ in $\Zl1(\R,\thT)$ for the action $\at$
intersects $(T_0)_2$ in a coset of $T_0\cap T_1$.
By Lemma \ref{l:cosets-T0-cap-T1}, each coset is an orbit for the action $\ath$ of $T\cap\Nt$ on $(T_0)_2$.
Thus the inclusion $(T_0)_2\into \Zl1(\R,\thT)$
induces a bijective map of orbit sets
\begin{equation}\label{e:bijection-orb}
\Zl1(\R,\thG)/T\,\isoto\,(T_0)_2/(T\cap\Nt)
\end{equation}
which sends a $T$-orbit $\Om\subseteq\Zl1(\R,\thT)$ to the $(T\cap\Nt)$-orbit $\Om'=\Om\cap(T_0)_2$.

The group $\NNt(\R)$ acts on $(T_0)_2$ by $\ath$ and on $\Zl1(\R,\thT)$ by $\at$\hs, or,
 which is the same, by $\ath$\hs, so that $(T_0)_2$ is an $\NNt(\R)$-stable subset of $\Zl1(\R,\thT)$.
Since $\NNt(\R)$ normalizes $T$ and $T\cap\Nt$, we have induced actions of $\NNt(\R)$
on the orbit sets $\Zl1(\R,\thG)/T$ and $(T_0)_2/(T\cap\Nt)$,
so that the bijection \eqref{e:bijection-orb} is $\NNt(\R)$-equivariant.

By Lemma \ref{l:N-M}
any orbit of $N_0$ in $\Zl1(\R,\thT)$ is of the form $\Qm=\NNt(\R)\ath\Om$
for some $T$-orbit $\Om\subseteq\Zl1(\R,\thT)$ and, similarly,
any orbit of $\Nt$ in $(T_0)_2$ is of the form
$\Qm'=\NNt(\R)\ath\Om'$ for some $(T\cap\Nt)$-orbit $\Om'\subseteq(T_0)_2$.
Furthermore, if $\Om'=\Om\cap(T_0)_2$, then $\Qm'=\Qm\cap(T_0)_2$.
We see that
for any $N_0$-orbit $\Qm$ in $\Zl1(\R,\thT)$ under $\at$,
the intersection $\Qm\cap(T_0)_2$ is an orbit of $\Nt$ under $\ath$\hs.
This yields the desired bijection between the orbit sets.
\end{proof}

\begin{remark}
Proposition \ref{p:T0-Z1-C} describes $\Ho1(\R,\thG)$ in terms of the  regular action $\ath$
of the complex algebraic group $\Nt$.
\end{remark}

\section{Shifting the action of the twisted normalizer}
\label{s:shifting}

\begin{subsec}
In order to compute the Galois cohomology, we need an explicit description of the quotient set $(T_0)_2/\Nt$.

Consider the semidirect product $\Ghat=G\rtimes\langle\htau\rangle$,
where $\langle\htau\rangle$ is the group of order 1 or 2
whose generator $\htau$ acts on $G$ as $\ttau$.
Then $\Ghat$ acts on its normal subgroup $G$ by conjugation.
In particular, the element
\[\tthat:=\tto\cdot\htau\hs\in\hs G\cdot\htau\subset\Ghat\]
acts on $G$ by $\inn(\tto)\circ\ttau=\theta$.
\end{subsec}

\begin{notation}
\label{n:t0-tau-z2}
Set
\begin{equation*}\label{e:zzz}
z=\tthat^2=(\tto\cdot\htau)^2=\tto^2\,\in Z(G^\der)\cap T_0\subset\Ghat.
\end{equation*}
Recall that $(T_0\cdot\htau)_2^z$ denotes the subset of elements with square $z$ in $T_0\cdot\htau$.
\end{notation}

\begin{remark}\label{r:t0-tau-z2-t0-z2}
If $t\in T_0$, then $(t\cdot\htau)^2=t^2$.
Since $z$ belongs to the finite group $Z(G^\der)$, $t$~is an element of finite order
whenever $t\cdot\htau\in(T_0\cdot\htau)_2^z$; in particular, $t\in\TT_0(\R)$ in this case.
Thus
\[(T_0\cdot\htau)^z_2=(\TT_0(\R)\cdot\htau)^z_2=\{t\cdot\htau\hs\mid\, t\in \TT_0(\R),\ t^2=z\}.\]
\end{remark}

\begin{lemma}\label{l:t0-tau-z2}
\begin{enumerate}
\renewcommand{\theenumi}{\roman{enumi}}
\item\label{i:shift}
The bijective map
\begin{equation}\label{e:times-tau}
T_0 \stackrel\sim\longrightarrow {T_0\cdot\htau},\quad
 s        \mapsto\hs       s\cdot\tthat=s\tto\cdot\htau,
\end{equation}
is $\Nt$-equivariant,
where $\Nt$ acts on $T_0$  by $\ath$, see \eqref{e:ath},
and on ${T_0\cdot\htau}$ by conjugation.
\item\label{i:sqroots-z}
The map \eqref{e:times-tau} restricts to an $\Nt$-equivariant bijection
between $(T_0)_2$ and $(T_0\cdot\htau)^z_2$.
\end{enumerate}
\end{lemma}

\begin{proof}
\eqref{i:shift} For $n\in \Nt$ and $s\in T_0$ we have
\[ n\hs s\hs\hs\theta(n)^{-1}\cdot\tthat=n\hs s\cdot \tthat\hs n^{-1}\tthat^{-1}\cdot \tthat
                =n\, (s\cdot\tthat)\, n^{-1}.\]
\eqref{i:sqroots-z} $(s\tto)^2=z$ if and only if $s^2=1$.
\end{proof}

\begin{corollary}[from Proposition \ref{p:T0-Z1-C} and Lemma \ref{l:t0-tau-z2}\eqref{i:sqroots-z}]
\label{c:times-tau}
The map
\[(T_0\cdot\htau)_2^z\to (T_0)_2\into \Zl1(\R,\thG),\quad t\cdot\htau\mapsto t\hs \tto^{-1}\]
 induces a bijection
$$(T_0\cdot\htau)_2^z\hs/\hs\Nt\isoto \Ho1(\R,\thG),$$
where $\Nt$ acts on $(T_0\cdot\htau)_2^z$ by conjugation.
\end{corollary}

\section{Lattices, restricted roots, and affine Weyl groups}
\label{s:lattices}

\begin{subsec}
In Section~\ref{s:logarithm} we shall reduce the computation of Galois cohomology
to computations in the Lie algebra of $\TT_0$.
To this end, we need some preparations.

Recall that $X=\X^*(T)$ and $X^\vee=\X_*(T)$ are the character
 and cocharacter lattices of $T$, respectively.
We regard $X$ as a lattice in $\tl^*$ and $X^\vee$ as a lattice in $\tl$.
A standard property of the exponential mapping implies the formula
\begin{equation}\label{e:exp-2-pi-i-z}
\nu(\exp 2\pi\ii z)=\exp(2\pi\ii\hssh d\nu(z))=\exp(2\pi\ii z\cdot\nu)\quad\text{for any }\nu\in X^\vee,\ z\in\C,
\end{equation}
where $\nu$ is regarded as a homomorphism $\GmC\to T$ on the left-hand side,
and as a vector in $\tl$, identified with $d\nu(1)$, on the right-hand side.

The scaled exponential mapping
\begin{equation*}
\Exp\colon\tl \longrightarrow T,\quad x \mapsto \exp(2\pi{x})
\end{equation*}
is a universal covering and a homomorphism from the additive group of $\tl$ to $T$ with kernel $\ii X^\vee$.
Note that $\ii X^{\vee}\subset\ttl(\R):=\Lie\TT(\R)$ and the restriction of $\Exp$
to $\ttl(\R)$ is the universal covering of the compact torus $\TT(\R)$.
\end{subsec}

\begin{subsec}
Set $X_0=\X^*(T_0)$.
Then the restriction homomorphism $X\to X_0$ is surjective.
We have a natural embedding $X_0\into\tl_0^*$.
Then $X_0\subset\tl_0^*$ is the projection of $X\subset\tl^*$
corresponding to the embedding $\tl_0\into \tl$.

We have $\tl_0=\Lie(T_0)=\tl^\ttau$.
Set
$X_0^\vee=\X_*(T_0)$.
Then $X_0^\vee=(X^\vee)^\tau$ is a lattice in $\tl_0$.
The restriction
$$
\Exp_0:\tl_0\longrightarrow T_0
$$
of the scaled exponential map $\Exp$ to $\tl_0$ is a universal covering
and a homomorphism with kernel $\ii X_0^\vee\subset\ttl_0(\R)=\Lie\TT_0(\R)$.
Its restriction to $\ttl_0(\R)$ is a universal covering of $\TT_0(\R)$.
\end{subsec}

\begin{subsec}
We have a canonical projection map
\[\tl\to\tl_0,\quad x\mapsto\half(x+\ttau(x)).\]
with kernel $\tl_1$.
We denote by $\widetilde{X}_0^{\vee}$ the projection of $X^{\vee}$ to $\tl_0$,
that is,
\[\widetilde{X}_0^{\vee}=\{\half(\nu+\ttau(\nu))\mid \nu\in X^\vee\subset \tl \}.\]
There are inclusions
\[X_0^{\vee}\subseteq\widetilde{X}_0^{\vee}\subseteq\half X_0^{\vee}.\]
\end{subsec}

\begin{lemma}\label{l:e0-T0-cap-T1}
The exponential map $\Exp_0$ sends $\ihalf X_0^{\vee}$ onto $(T_0)_2$
and $\ii\widetilde{X}_0^{\vee}$ onto ${T_0\cap T_1}$, and thus
induces an isomorphism
\[\half X_0^{\vee}/X_0^{\vee}\simeq\ihalf X_0^{\vee}/\ii X_0^{\vee}\isoto(T_0)_2\hs,\]
which restricts to an isomorphism  $\widetilde{X}_0^{\vee}/X_0^\vee\isoto T_0\cap T_1$.
\end{lemma}

\begin{proof}
It suffices to show that $\Exp(\ii\widetilde{X}_0^\vee)=T_0\cap T_1$.

Let $\nu\in \widetilde{X}_0^\vee$. Then $\ii \nu\in \tl_0$; hence $\Exp(\ii \nu)\in T_0$.
We have $\nu=\half(\nu'+\ttau(\nu'))$ for some $\nu'\in X^\vee$.
Set $y=\ihalf(\ttau(\nu')-\nu')\in\tl_1$; then $\ii \nu=\ii \nu'+y$.
We see that
\[\Exp(\ii \nu)=\Exp(\ii \nu')\cdot \Exp(y)=1\cdot \Exp(y)\in T_1\hs.\]
Thus $\Exp(\ii \nu)\in T_0\cap T_1$.

Conversely, let $t\in T_0\cap T_1$\hs.
Write  $t=\Exp(x)$ with $x\in\tl_0$\hs, and $t=\Exp(y)$ with $y\in\tl_1$; then  $x=y+\ii \nu'$
for some $\nu'\in X^\vee$.
Since $\ttau(x)=x$ and $\ttau(y)=-y$, we get
\[x=\ihalf(\nu'+\ttau(\nu'))\in\ii\widetilde{X}_0^\vee\hs.\]
Thus $t=\Exp(x)\in \Exp(\ii\widetilde{X}_0^\vee)$\hs,
as required.
\end{proof}

\begin{subsec}
In consideration of roots and Weyl groups, we may reduce most of issues to the case of a semisimple group.
For simplicity, we assume from now till the end of Section~\ref{s:lattices} that $G$ is \emph{semisimple}.
Recall that $G^\ad$ denotes the adjoint group of $G$ and $G^\ssc$ is the universal cover of $G$.
Also, $T^{\ad}$ and $T^\ssc$ are the compatible with $T$ maximal tori in $G^{\ad}$ and $G^\ssc$, respectively.
We identify the Lie algebras of $T^{\ad}$ and $T^\ssc$ with $\tl=\Lie T$.

As usual, we write
$$P=\X^*(T^\ssc)\quad\text{and}\quad Q=\X^*(T^\ad);$$
these are the weight lattice and the root lattice.
Then $P$ and $Q$ naturally embed into $\tl^*$, and we obtain three lattices in $\tl^*$:
\[Q\subseteq X\subseteq P.\]
Also we write
$$P^\vee=\X_*(T^\ad)\quad\text{and}\quad Q^\vee=\X_*(T^\ssc);$$
these are the coweight lattice and the coroot lattice.
Then $P^\vee$ and $Q^\vee$ naturally embed into $\tl$, and we obtain three  lattices in $\tl$\,:
\[ P^\vee\supseteq X^\vee\supseteq Q^\vee. \]
Note that the lattices $P^\vee$ and $Q^\vee$ are dual to $Q$ and $P$, respectively.

We set
\[C= P^\vee/ Q^\vee,\quad {F}= X^\vee/ Q^\vee.\]
Then ${F}\subseteq C$ are finite abelian groups.
The scaled exponential map $\Exp^\ssc:\tl\to T^\ssc$
with kernel $\ii Q^\vee$ induces isomorphisms
\begin{align}\label{e:C-Z}
&C\longisoto\ker[T^\ssc\to T^\ad]=Z(G^\ssc)=\pi_1(G^\ad),\\
&F\longisoto\ker[T^\ssc\to T]=\pi_1(G).\notag
\end{align}
Thus ${F}$ is the fundamental group of $G$.
\end{subsec}

\begin{subsec}\label{ss:aff-W}
Recall that $W=\sN_G(T)/T$ is the Weyl group of $(G,T)$,  of $(G^\ssc,T^\ssc)$, and of $(G^\ad,T^\ad)$.
The scaled exponential maps from $\tl$ to $T$, to $T^\ssc$, and to $T^\ad$ are $W$-equivariant.
Since $W$ acts trivially on $Z(G^\ssc)$, we see from the $W$-equivariant
isomorphism \eqref{e:C-Z} that $W$ acts trivially on $C= P^\vee/ Q^\vee$.

We set
\[\Wtil= X^\vee\rtimes W,\quad \Wtil^\ad= P^\vee\rtimes W, \quad
\Wtil^\ssc= Q^\vee\rtimes W;\]
then $\Wtil^\ssc\subseteq \Wtil\subseteq \Wtil^\ad$.
Since $W$ acts on $P^\vee/ Q^\vee$ trivially, one checks
that $\Wtil^\ssc$ is a  normal subgroup of $\Wtil^\ad$ and
\[ \Wtil^\ad/\Wtil^\ssc\simeq C,\quad\Wtil/\Wtil^\ssc\simeq{F}.\]

The lattices $Q^\vee$, $X^\vee$, $P^\vee$ act on $\tl$ by translations:
\begin{equation}\label{e:X-action}
\nu\colon\ x\mapsto x+\ii \nu\quad  \text{for } \nu\in P^\vee,\ \ x\in\tl.
\end{equation}
Combined with the natural linear action of $W$ on $\tl$, the action \eqref{e:X-action}
gives rise to actions of $\Wtil^\ssc$, $\Wtil$, and  $\Wtil^\ad$ on $\tl$ by affine transformations.
Each of these actions preserves $\ttl(\R)$ (for which we added the multiplier $\ii$ in \eqref{e:X-action})
and restricts to an action on $\ttl(\R)$ by affine isometries preserving
the Euclidean structure induced by the Killing form on  $\gl$.
The group $\Wtil^\ssc$ regarded as a group of motions of the Euclidean space
$\ttl(\R)$ is a crystallographic group generated by reflections,
known as the \emph{affine} (or \emph{extended}) \emph{Weyl group} of the root system $\Rt=\Rt(G,T)$;
see Bourbaki \cite[Section VI.2.1]{Bourbaki},
Gorbatsevich, Onishchik, and Vinberg \cite[Section 3.3.6, Proposition 3.10(1)]{OV2},
and also Section~\ref{s:inner} below for details.
\end{subsec}

\begin{subsec}
For any character $\lambda\in X$, let $\bar\lambda\in X_0$ denote the restriction of $\lambda$ to $T_0$.
In particular, we consider the restricted roots $\bar\alpha$ for all $\alpha\in\Rt$.

Let $\gl^{\pm\ttau}$ denote the $(\pm1)$-eigenspaces for $\ttau$ in $\gl$.
Then $\gl^\ttau$ is the Lie algebra of a reductive (in fact, semisimple) subgroup $G^{\ttau}\subseteq G$
and $\gl^{-\ttau}$ is preserved under the adjoint action of $G^{\ttau}$.

If a root $\alpha\in\Rt$ is $\tau$-invariant, then $\bar\alpha$ is either a root of $\gl^\ttau$
or a weight of $\gl^{-\ttau}$ with respect to $T_0$,
depending on whether $\gl_{\alpha}\subset\gl^\ttau$ or $\gl_{\alpha}\subset\gl^{-\ttau}$, respectively.
If $\alpha\in\Rt$ is not $\tau$-invariant, then $\bar\alpha=\overline{\tau(\alpha)}$
is both a root of $\gl^\ttau$ and a weight of $\gl^{-\ttau}$
with eigenvectors $\e\alpha\pm\e{\tau(\alpha)}$, respectively,
where $\e\alpha$ is any generator of $\gl_\alpha$  and $\e{\tau(\alpha)}=\ttau(\e\alpha)\in \gl_{\tau(\alpha)}$.
\end{subsec}

\begin{proposition}\label{p:res-rt}\
\begin{enumerate}
\renewcommand{\theenumi}{\roman{enumi}}
\item\label{i:res-rt} The set $\Rtbar=\{\bar\alpha\mid\alpha\in\Rt\}$ is a root system in $X_0$ (possibly non-reduced).
\item\label{i:res-sm} The set $\Smbar=\{\bar\alpha\mid\alpha\in\Sm\}$ is a basis of $\Rtbar$.
\item\label{i:res-Weyl} The Weyl group $W(\Rtbar)$ of $\Rtbar$ is naturally isomorphic to $W_0$.
\end{enumerate}
\end{proposition}

\begin{proof}
For \eqref{i:res-rt} and \eqref{i:res-sm} see Gorbatsevich, Onishchik, and Vinberg \cite[Section 3.3.9]{OV2},
in particular, Theorem~3.14(3) of loc.~cit., or Timashev \cite[Lemma 26.8]{Timashev}.
A proof of \eqref{i:res-Weyl} is implicit in \cite[Section 3.3.9]{OV2}
for the case of simple $G$; we provide a direct general argument.

Note that $W_0\subseteq W$ acts on $T$ preserving $T_0$, and the induced action of $W_0$ on $X$ preserves $\Rt$.
Hence $W_0$ can be identified with a group of automorphisms of $X_0$ preserving $\Rtbar$.
Each restricted simple root $\bar\alpha$ ($\alpha\in\Sm$)
is a root of $\gl^{\tau}$ with respect to $T_0$, whose root subspace is spanned by $\e\alpha+\e{\tau(\alpha)}$,
where, since  $\ttau$ respects the pinning $(\e\beta)$, we have  $\e\alpha+\e{\tau(\alpha)}=\e\alpha+\tau(\e\alpha)\in\gl^\ttau$.
Hence each simple reflection in $W(\Rtbar)$ and, furthermore, each element of $W(\Rtbar)$,
can be represented by some $n\in\Nm_{G^{\ttau}}(T_0)=N_0^{\ttau}$.
Therefore $W(\Rtbar)\subseteq W_0$.

On the other hand, each $w\in W_0$ induces an automorphism of $\Rtbar$;
hence the set $w(\Smbar)$ is a basis of $\Rtbar$.
Then $w(\Smbar)=w'(\Smbar)$ for some $w'\in W(\Rtbar)$, and $w''=w^{-1}w'$ preserves $\Smbar$.
We claim that $w''$ preserves $\Sm$, too. Indeed, if $\alpha\in\Sm$ and $\beta=w''(\alpha)$,
then $\bar\beta=w''(\bar\alpha)\in\Smbar$; hence $\bar\beta=\bar\gamma$ for some $\gamma\in\Sm$.
Since the eigenspaces of $T_0$ in $\gl^{\pm\tau}$ are one-dimensional by \cite[Section 3.3.9, Theorem 3.14(i)]{OV2},
it follows from the description of these eigenspaces given right before Proposition~\ref{p:res-rt}
that $\beta=\gamma$ or $\beta=\tau(\gamma)\in\Sm$. This proves our claim.
As $w''\in W$ preserves $\Sm$, we have $w''=1$, whence $w=w'\in W(\Rtbar)$.
Therefore $W_0\subseteq W(\Rtbar)$.
\end{proof}

\begin{subsec}\label{ss:s-res-rt-wt}
Denote by $T_0^\ad$ and $T_0^\ssc$
the subtori in $T^\ad$ and in $T^\ssc$,  respectively, with Lie algebra $\tl_0$\hs.
We set
$$Q_0=\X^*(T_0^\ad)\quad\text{and}\quad P_0=\X^*(T_0^\ssc).$$
We obtain three lattices  in $\tl_0^*$:
\[ Q_0\subseteq X_0\subseteq P_0\hs.\]
Note that $Q_0$ and $P_0$ are the projections to $\tl_0^*$ of $Q$ and $P$, respectively.

We also set
\[P_0^\vee=\X_*(T_0^\ad)=(P^\vee)^\tau\quad\text{and}\quad Q_0^\vee=\X_*(T_0^\ssc)=(Q^\vee)^\tau.\]
These lattices in $\tl_0$ are dual to $Q_0$ and $P_0$, respectively.

Choose simple roots $\alpha_1,\dots,\alpha_\ell\in\Sm$ which are representatives of the $\tau$-orbits in~$\Sm$.
Let $\alpha_i^{\vee}\in\Sm^\vee$, $\omega_i\in P$, and $\omega_i^\vee\in P^\vee$ denote the simple coroot,
fundamental weight and coweight corresponding to $\alpha_i$ ($i=1,\dots,\ell$), respectively.
\end{subsec}

\begin{lemma}\label{l:Q0P0}\
\begin{enumerate}
\renewcommand{\theenumi}{\roman{enumi}}
\item\label{i:Q0P0} $Q_0$ and $P_0$ are the root lattice and the weight lattice of $\Rtbar$, respectively.
The restricted simple roots $\bar\alpha_1,\dots\bar\alpha_\ell$ comprise the basis $\Smbar$ of $Q_0$
and the restricted fundamental weights $\bar\omega_1,\dots\bar\omega_\ell$ comprise a basis of $P_0$.
\item\label{i:dual-Q0P0} $Q_0^\vee$ and $P_0^\vee$ are the coroot lattice
and the coweight lattice of $\Rtbar$, respectively.
The simple coroots of $\Rtbar$ with respect to $\Smbar$, which comprise a basis
of $Q_0^\vee$, are $\bar\alpha_j^\vee=\alpha_j^\vee$
 and $\bar\alpha_k^\vee=\alpha_k^\vee+\tau(\alpha_k^\vee)$, where $j,k$ run over all integers $1,\dots,\ell$
 such that $\alpha_j$ is fixed and $\alpha_k$ is moved by $\tau$, respectively.
 The fundamental coweights of $\Rtbar$ with respect to $\Smbar$, which comprise a basis of $P_0^\vee$,
 are $\bar\omega_j^\vee=\omega_j^\vee$ and $\bar\omega_k^\vee=\omega_k^\vee+\tau(\omega_k^\vee)$, with the same notation.
\end{enumerate}
\end{lemma}

\begin{proof}
The claims about the bases of $Q_0$, $P_0$, $Q_0^\vee$, and $P_0^\vee$ are clear.
It is also clear that $Q_0$ is the root lattice of $\Rtbar$ and $\bar\alpha_i$ are the simple roots.
It follows that the dual lattice $P_0^\vee$ is the coweight lattice
of $\Rtbar$ and the vectors $\omega_j^\vee$ and $\omega_k^\vee+\tau(\omega_k^\vee)$,
which comprise the dual basis of $P_0^\vee$
for $\{\bar\alpha_1,\dots\bar\alpha_\ell\}$, are the fundamental coweights.

Under the identification of $\tl_0$ with $\tl_0^*$ via the Killing form,
the simple coroots of $\Rtbar$ are proportional to the simple roots $\bar\alpha_i$ of $\Rtbar$,
which are the projections of $\alpha_i\in\tl^*$ to $\tl_0^*$.
Similarly, the simple coroots $\alpha_i^\vee$ of $\Rt$ are proportional to $\alpha_i$. It follows that
the vectors $\alpha_j^\vee$ and $\alpha_k^\vee+\tau(\alpha_k^\vee)$ in $Q_0^\vee$
are proportional to the simple coroots of~$\Rtbar$.
To show that these are indeed the simple coroots,
it suffices to check that the pairing of such a vector with the respective simple root $\bar\alpha_i$ ($i=j$ or $k$)
is equal to $2$ unless $2\bar\alpha_i\in\Rtbar$, in which case the pairing should be equal to $1$.

We have $\langle\bar\alpha_j,\alpha_j^\vee\rangle=\langle\alpha_j,\alpha_j^\vee\rangle=2$.
Note that $2\bar\alpha_j\notin\Rtbar$.
Indeed, if $2\bar\alpha_j=\bar\alpha$ for some $\alpha\in\Rt$,
then $\langle\alpha,\alpha_j^\vee\rangle=4$, which is not possible.

Also we have
\[\lb\bar\alpha_k,\alpha_k^\vee+\tau(\alpha_k^\vee)\rb=\lb\alpha_k,\alpha_k^\vee+\tau(\alpha_k^\vee)\rb=2\]
if $\alpha_k$ and $\tau(\alpha_k)$ are not linked in $\D(G)$.
In this case we have $2\bar\alpha_k\notin\Rtbar$.
Indeed, if $2\bar\alpha_k=\bar\beta$ for some $\beta\in\Rt$, then
$\bar\alpha_k=\overline{\tau(\alpha_k)}=\bar\beta/2$ and
$$\lb\beta,\alpha_k^\vee+\tau(\alpha_k^\vee)\rb= \lb\alpha_k+\tau(\alpha_k),\hs\alpha_k^\vee+\tau(\alpha_k^\vee)\rb=4,$$
which may happen only if $\alpha_k=\beta=\tau(\alpha_k)$, a contradiction.

Finally, if $\alpha_k$ and $\tau(\alpha_k)$ are linked in $\D(G)$, then $\lb\alpha_k,\tau(\alpha_k^\vee)\rb =-1$.
In this case we have  $2\bar\alpha_k=\bar\beta$ for $\beta=\alpha_k+\tau(\alpha_k)\in\Rt$ and
\[\lb\bar\alpha_k,\alpha_k^\vee+\tau(\alpha_k^\vee)\rb=\lb\alpha_k,\alpha_k^\vee+\tau(\alpha_k^\vee)\rb=1.\]

Since the lattice  $Q_0^\vee$ is generated by the simple coroots $\alpha_j^\vee$
and $\alpha_k^\vee+\tau(\alpha_k^\vee)$, it is the coroot lattice of $\Rtbar$.
Hence the dual lattice $P_0$  is the weight lattice of $\Rtbar$.  The vectors $\bar\omega_i$,
which comprise the basis of $P_0$ dual to the basis of $Q_0^\vee$
consisting of the simple coroots $\alpha_j^\vee$, $\alpha_k^\vee+\tau(\alpha_k^\vee)$,
are the fundamental weights.
\end{proof}

\begin{subsec}
Let $\wt P^\vee_0$ and $\wt Q^\vee_0$ denote the projections of $P^\vee$ and $Q^\vee$ to $\tl_0$, respectively.
We have three lattices in $\tl_0$:
\[\wt Q^\vee_0\subseteq \wt X^\vee_0\subseteq \wt P^\vee_0\hs.\]
A basis of $\wt Q_0^\vee$ consists of all vectors $\bar\alpha_j^\vee$ and $\half\bar\alpha_k^\vee$,
and a basis of $\wt P_0^\vee$ consists of all vectors $\bar\omega_j^\vee$ and $\half\bar\omega_k^\vee$,
with the notation of Lemma~\ref{l:Q0P0}\eqref{i:dual-Q0P0}.

We set
\[C_0=\wt P^\vee_0/\wt Q^\vee_0,\quad {F}_0=\wt X^\vee_0/\wt Q^\vee_0;\]
then ${F}_0\subseteq C_0$.
The canonical epimorphisms
\[ P^\vee\to\wt P^\vee_0,\quad X^\vee\to \wt X^\vee_0,\quad Q^\vee\to\wt Q^\vee_0\]
induce epimorphisms
\[{F}\to {F}_0\quad\text{and}\quad C\to C_0\hs.\]

The lattices $\wt X_0^\vee$, $\wt P_0^\vee$, and $\wt Q_0^\vee$ are preserved by the action of $W_0$ on $\tl_0$
and, similarly to Subsection~\ref{ss:aff-W}, we define
\[\Wotil=\wt X^\vee_0\rtimes W_0,\quad \Wotil^\ad=\wt P^\vee_0\rtimes W,\quad \Wotil^\ssc=\wt Q^\vee_0\rtimes W,\]
so that $\Wotil^\ssc\subseteq\Wotil\subseteq\Wotil^\ad$,
 where the subgroup $\Wotil^\ssc$ is normal in $\Wotil^\ad$, and
\[\Wotil^\ad/\Wtil^\ssc\simeq C_0,\quad\Wotil/\Wotil^\ssc\simeq F_0.\]

The action \eqref{e:X-action} restricts to an action of $\wt X_0^\vee$,
$\wt P_0^\vee$, and $\wt Q_0^\vee$ on $\ttl_0(\R)$ by translations,
which, in turn, gives rise to an action of $\Wotil$, $\Wotil^\ad$, and $\Wotil^\ssc$ on $\ttl_0(\R)$ by affine isometries.
\end{subsec}

\begin{lemma}[{\cite[Section 3.3.10, Proposition 3.15(1)]{OV2}}]
The group $\Wotil^\ssc$  regarded as a group of motions of the Euclidean space $\ttl_0(\R)$
is a crystallographic group generated by reflections along the hyperplanes
$\{\bar\alpha(x)=\ii k\}$ and $\{\bar\beta(x)=\ii k/2\}$ ($k\in\Z$)
over all roots $\bar\alpha$ of $\gl^\ttau$ and all weights $\bar\beta$ of $\gl^{-\ttau}$.
\end{lemma}

The group $\Wotil^\ssc$ is known as the \emph{affine Weyl group} associated with $\tau$;
see Onishchik and Vinberg \cite[Section 4.4.5]{OV},
Gorbatsevich, Onishchik, and Vinberg \cite[Section 3.3.10]{OV2},
and also Sections~\ref{s:outer}--\ref{s:nas} below for details.

\section{Logarithm of the action of the twisted normalizer}
\label{s:logarithm}

\begin{subsec}
We wish to describe $(T_0\cdot\htau)_2^z\hs/\hs \Nt=(\TT_0(\R)\cdot\htau)_2^z\hs/\hs \Nt$, where
the group $\Nt$ acts on the $\C$-variety  $T_0\cdot\htau$ by conjugation:
\begin{align}\label{e:Ntau-action}
n\colon\   t\cdot\htau\mapsto n\hs(t\cdot\htau)\hs n^{-1}&=n\hs t\hs\ttau(n)^{-1}\cdot\htau\\ \notag
    &=n\hs t\hs n^{-1}\cdot n\hs \ttau(n)^{-1}\cdot\htau
\quad\text{for }n\in \Nt,\, t\in T_0\hs.
\end{align}

\begin{lemma}\label{l:Ntau-action}
The kernel of the action \eqref{e:Ntau-action}  is $T^\ttau=T_0\cdot(T_1)_2$.
\end{lemma}

\begin{proof}
Let $n\in \Nt\subset N_0$ be in the kernel of the action \eqref{e:Ntau-action}.
 Substituting $t=1$, we obtain that $\ttau(n)=n$.
It follows that for all $t\in T_0$ we have $ntn^{-1}=t$, and hence $n\in T$.
Since  $\ttau(n)=n$, we conclude that $n\in T^\ttau$.

For $t=t_0t_1\in T$ with $t_0\in T_0,\,t_1\in T_1$, the condition $\ttau(t)=t$
means that $t_0t_1^{-1}=t_0t_1$, that is, $t_1^2=1$.  It follows that $T^\ttau=T_0\cdot(T_1)_2$.
Thus the kernel of the action \eqref{e:Ntau-action}  is contained in  $T^\ttau=T_0\cdot(T_1)_2$.
The inverse inclusion is clear.
\end{proof}

\begin{notation}
We define
$\Wohat=\Nt/T^\ttau.$
\end{notation}
By Lemma \ref{l:Ntau-action}, the action \eqref{e:Ntau-action}
induces an effective action of $\Wohat$ on $T_0\cdot\htau$.
The inclusion map $\Nt\into N_0$ induces a homomorphism $\Nt\to W_0$
(which is surjective by Lemma \ref{l:N-M}) with kernel $\Nt\cap T$.
Since $T^\ttau\subseteq \Nt\cap T$, we obtain an induced homomorphism
\[\Wohat\to W_0\]
with kernel $(\Nt\cap T)/T^\ttau$.
We have
\begin{align*}
\Nt\cap T&=\{t=t_0t_1\in T\mid t_0\in T_0,\ t_1^2\in (T_0)_2\},\\
T^\ttau&=\{t=t_0t_1\in T\mid t_0\in T_0,\ t_1^2=1\}.
\end{align*}
The homomorphism
\[ \Nt\cap T\to T_1, \quad t=t_0t_1\mapsto t\cdot\ttau(t)^{-1}=t_1^2\]
with kernel $T^\ttau$ and with image $T_0\cap T_1$ yields
an isomorphism
\[ (\Nt\cap T)/T^\ttau\isoto T_0\cap T_1\hs.\]
The inverse isomorphism is given as follows:
\begin{equation}\label{e:sqrt}
T_0\cap T_1\isoto(\Nt\cap T)/T^\ttau,\quad s\mapsto
\sqrt{s}\cdot T^\ttau\quad\text{ for }s\in T_0\cap T_1\hs,
\end{equation}
where $\sqrt{s}$ is any element $t_1\in T_1$ such that $t_1^2=s$.
Thus we obtain a short exact sequence
\begin{equation}\label{e:exact-What}
1\to T_0\cap T_1\to \Wohat\to W_0\to 1,
\end{equation}
where the homomorphism $T_0\cap T_1\to\Wohat$ is given by \eqref{e:sqrt}.

We construct a canonical splitting of \eqref{e:exact-What}.
Consider the subgroup
\[ N_0^\ttau=\{n\in N_0\mid \ttau(n)=n\}\subseteq\Nt.\]
The inclusion map $N_0^\ttau\into N_0$ induces a homomorphism
\begin{equation}\label{e:N-0-tau}
N_0^\ttau\to W_0\hs,
\end{equation}
with kernel  $N_0^\ttau\cap T=T^\ttau$.
\end{subsec}

\begin{lemma}\label{l:N-0-tau}
The homomorphism \eqref{e:N-0-tau} is surjective.
\end{lemma}

\begin{proof}
By Proposition~\ref{p:res-rt}\eqref{i:res-Weyl}, the group $W_0$ regarded as a group of automorphisms of $X_0$
is spanned by reflections along the restricted simple roots $\bar\alpha\in\Smbar$,
which also comprise a set of simple roots for $G^\ttau$ with respect to $T_0$.
Hence each element $w\in W_0$ can be represented by some $n\in\Nm_{G^{\ttau}}(T_0)=N_0^{\ttau}$.
\end{proof}

\begin{subsec}
By Lemma \ref{l:N-0-tau}, the injective homomorphism  $N_0^\ttau/T^\ttau\to W_0$
is surjective, and hence is an isomorphism.
We obtain a canonical splitting
\[ W_0\isoto  N_0^\ttau/T^\ttau\into \Nt/T^\ttau=\widehat{W}_0\hs,\]
which defines a canonical isomorphism
\[(T_0\cap T_1)\rtimes W_0\,\isoto\, \Wohat.\]

We describe the action of $\Wohat=(T_0\cap T_1)\rtimes W_0$ on $T_0\cdot\htau$.
An element $s\in T_0\cap T_1$ acts by
\begin{equation*}
s\colon\ t\cdot\htau\ \mapsto \sqrt{s}\hs (t\cdot\htau)\hs (\sqrt{s})^{-1}=st\cdot\htau
\quad\text{for } t\in T_0.
\end{equation*}
The group $W_0$ acts by
\[w\colon\ t\cdot\htau\mapsto n\hs(t\cdot\htau)\hs n^{-1}=n\hs t\hs n^{-1}\cdot\htau=w(t)\cdot\htau
     \quad\text{for }w\in W_0,\ t\in T_0\,,\]
where $n\in N_0^\ttau$ is a representative of $w$.
\end{subsec}

\begin{subsec}\label{e:e-e0}
By Lemma~\ref{l:e0-T0-cap-T1} the scaled exponential mapping $\Exp$ restricts to a surjective homomorphism
\begin{equation}\label{e:e0-wt-X}
\Exp_0\colon\, \ii\widetilde{X}_0^{\vee}\,\to\, T_0\cap T_1\hs.
\end{equation}
We construct a map
\begin{equation}\label{e:lambda-vs}
\Exp_W\colon\,\Wotil\to\Wohat,\quad \nu\cdot w\mapsto \Exp(\ii \nu)\cdot w\quad
\text{for }\nu\in \widetilde{X}_0^{\vee}\hs,\, w\in W_0\hs,
\end{equation}
Since the homomorphism \eqref{e:e0-wt-X} is $W_0$-equivariant, the map \eqref{e:lambda-vs} is a homomorphism.
It is clear that $\Exp_W$ is surjective and its kernel is $X_0^\vee$.
\end{subsec}

\begin{lemma}\label{l:Wtil-What}
The map
\begin{equation}\label{e:ttt0-e0}
 \Exphat\colon\ttl_0(\R)\to \TT_0(\R)\cdot\htau,\quad y\mapsto \exp(2\pi y)\cdot\htau \quad\text{for }y\in\ttl_0(\R)
 \end{equation}
is compatible with the homomorphism $\Exp_W\colon\Wotil\to\Wohat$ and thus
induces a bijection of the  orbit sets
\begin{equation}\label{e:ttt0-Wtil}
\Exp_\orb\colon \ttl_0(\R)/\Wotil\isoto (\TT_0(\R)\cdot\htau)/\Wohat\hs.
\end{equation}
\end{lemma}

\begin{proof}
It is routine to check that $\Exphat$ is $\Exp_W$-equivariant, that is,
\[\Exphat(\wtil\cdot y)=\Exp_W(\wtil)\cdot \Exphat(y)\quad \text{for }\wtil\in\Wotil,\ y\in \ttl_0(\R).\]
Hence the map on orbits
$\Exp_\orb$ of \eqref{e:ttt0-Wtil}  is well defined.
Since the map $\Exphat$ is surjective, the map $\Exp_\orb$ is surjective as well.

We show that the map $\Exp_\orb$ is injective.
Let $\Wotil\cdot y_1$ and $\Wotil\cdot y_2$ be two orbits in $\ttl_0(\R)$
with the same image in $(T_0(\R)\cdot\htau)/\Wohat$.
This means that
\[\Exphat(y_2)=\what\cdot \Exphat(y_1)\quad\text{for some }\what\in\Wohat.\]
Since the homomorphism $\Exp_W$ is surjective, there exists $\wtil'\in\Wotil$ such that $\what=\Exp_W(\wtil')$.
Set $y_2'=\wtil'\cdot y_1$; then $\Exphat(y_2')=\Exphat(y_2)$.
It follows that
\[y_2=\ii \nu +y_2'\quad\text{for some }\nu\in X_0^\vee\hs.\]
Set $\wtil=\nu\cdot \wtil'\in \Wotil\hs.$
Then $y_2=\wtil\cdot y_1$.
Thus $\Wotil\cdot y_1=\Wotil\cdot y_2$.
We have proved that the map $\Exp_\orb$ of \eqref{e:ttt0-Wtil}
is injective, and hence bijective, as required.
\end{proof}

\begin{subsec}
Suppose now that $\GG$ is semisimple. Recall that the group $\Wotil=\widetilde{X}_0^{\vee}\rtimes W_0$
contains a subgroup of finite index $\Wotil^\ssc=\widetilde{Q}_0 ^{\vee}\rtimes W_0$\hs.
We have a finite group
 \[{F}_0=\widetilde{X}_0^\vee/\widetilde{Q}_0^\vee\simeq\Wotil/\Wotil^\ssc\]
and  a canonical bijection of the orbit sets
\[(\ttl_0(\R)/\Wotil^\ssc)\hs/{F}_0\hs\isoto\hs\ttl_0(\R)/\Wotil\hs,\]
and thus a chain of bijections
\begin{equation*}
 (\ttl_0(\R)/\Wotil^\ssc)\hs/{F}_0\hs\isoto\hs\ttl_0(\R)/\Wotil\isoto (\TT_0(\R)\cdot\htau)/\hs\Wohat\hs.
\end{equation*}
In order to compute $(\TT_0(\R)\cdot\htau)/\hs\Wohat$,
we wish to compute $(\ttl_0(\R)/\Wotil^\ssc)\hs/{F}_0$.
Note that the group
\[C_0= \wt P_0^\vee/\wt Q_0^\vee\simeq  \Wotil^\ad/\Wotil^\ssc   \]
naturally acts on the orbit set  $\ttl_0(\R)/\Wotil^\ssc$,
and the group ${F}_0$ acts on $\ttl_0(\R)/\Wotil^\ssc$ via the embedding
${F}_0\into C_0\hs$.
So we wish to describe the set $\ttl_0(\R)/\Wotil^\ssc$ together with the action of $C_0$,
where we may and shall assume that $\GG$ is simply connected.
\end{subsec}

\begin{subsec}
A pair $(\GG,\theta)$ with simply connected $\GG$ naturally decomposes into a direct product
of {\em indecomposable} pairs.
It suffices to describe the set $\ttl_0(\R)/\Wotil^\ssc$ together with the action of $C_0$
for any indecomposable pair $(\GG,\theta)$.
There are three cases:
\begin{enumerate}
\renewcommand{\theenumi}{\alph{enumi}}
  \item\label{i:inner-0} $\GG$ is simple and $\theta$ is an {\em inner} involution of $\GG$;
  \item\label{i:outer-0} $\GG$ is simple and $\theta$ is an {\em outer} involution of $\GG$;
  \item\label{i:swap-0}  $\GG^=\GG^{\prime}\times_\R \GG^{\prime\prime}$, where
  $\theta$ swaps the isomorphic simple factors $\GG^{\prime}$ and $\GG^{\prime\prime}$.
\end{enumerate}
In the next three sections, for an  indecomposable pair $(\GG,\theta)$ of each of the types
\eqref{i:inner-0}, \eqref{i:outer-0}, and \eqref{i:swap-0} above,
we shall describe the orbit set $\ttl_0(\R)/\Wotil^\ssc $ and the action
of $C_0$ on it.
\end{subsec}

\section{Case of an inner form of a simple compact group}
\label{s:inner}

\begin{subsec}\label{ss:ext-inner}
In this section $\GG$ is a simply connected, simple, compact  $\R$-group,
and the involution $\theta$ of $\GG$ is {\em inner}, that is, $\tau=\id$.
Then $\TT_0=\TT$, $\ttl_0(\R)=\ttl(\R)$, $\Wotil^\ssc=\Wtil^\ssc$, $C_0=C$, etc.;
see Section~\ref{s:lattices} for notation.

Let $\Dtil=\Dtil(G,T,B)=\Dtil(\Rt,\Sm)$  denote the extended Dynkin diagram;
the set of vertices of $\Dtil$ is $\Smtil=\{\alpha_0,\alpha_1,\dots,\alpha_\ell\}$,
where $\alpha_1,\dots,\alpha_\ell$ are the simple roots
enumerated as in \cite[Table~1]{OV}, and $\alpha_0$ is the lowest root.
These roots $\alpha_0,\alpha_1,\dots,\alpha_\ell$ are linearly dependent, namely,
\begin{equation}\label{e:sum-mj}
m_0\alpha_0+m_1\alpha_1+\dots+m_{\ell}\hs\alpha_\ell=0,
\end{equation}
where the coefficients $m_i$ are positive integers for all $i=0,1,\dots,\ell$, and
$m_0=1$.
These coefficients $m_i$ are tabulated in \cite[Table 6]{OV} and in \cite[Table 3]{OV2};
see also Table \ref{t:tab-nt} below.
In the diagrams on the left,
we write only the coefficients $m_i$ that are $\le2$.
In the diagrams on the right, we write the coefficients $c_i$ (modulo $\Z$) of the decomposition
into a linear combination of simple roots $\alpha_i$
for a representative ($\omega_1\hs,\omega_\ell\hs,$ etc.) of a generator of the group $P/Q$.
We write only the coefficients $c_i$ for the simple roots $\alpha_i$ for which $m_i\le 2$.
The vertex corresponding to $\alpha_0$ is painted in black.
\end{subsec}

\begin{subsec}\label{ss:barycentric-inner}
Following \cite[Section 3.3.6]{OV2}, we introduce the {\em barycentric coordinates}
$x_0$, $x_1$, \dots , $x_{\ell}$ of a point  $x\in \ttl(\R)$ by setting
\begin{equation}\label{e:barycentric}
\alpha_i(x)=\ii x_i\quad\text{for }i=1,\dots,\ell,\quad \alpha_0(x)=\ii(x_0-1).
\end{equation}

It follows from \eqref{e:sum-mj} that
\begin{equation*}\label{e:app-sum-2}
m_0x_0+m_1x_1+\dots+m_{\ell}x_{\ell}=1.
\end{equation*}
By Bourbaki \cite[Section VI.2.2, Proposition 5(i)]{Bourbaki},
see also \cite[Section 3.3.6, Proposition 3.10(2)]{OV2},
the closed simplex $\Delta\subset\ttl(\R)$ given by the inequalities
\begin{equation*}
x_0\ge 0,\ x_1\ge 0,\ \dots,\ x_\ell\ge 0
\end{equation*}


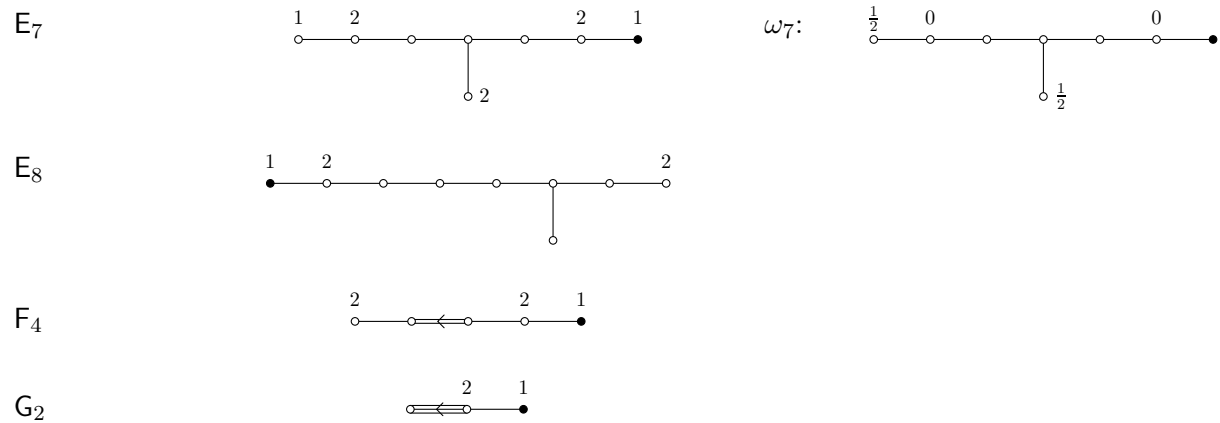
\tikzset{/Dynkin diagram,mark=o,affine mark=*,
edge length=0.75cm,arrow shape/.style={-{angle 90}}}

\begin{landscape}
\begin{longtable}[c]{l@{\qquad\qquad}c@{\qquad\quad}l@{\qquad}c}
\caption{Coefficients $m_i$ and $c_i$ on extended Dynkin diagrams}
\label{t:tab-nt}
\\[7ex]
\endfirsthead
\caption{(continued)}
\\[7ex]
\endhead
\endfoot
$\AAA_1$
&\dynkin[edge length=0.75cm,
labels*={1,1}
]%
A[1]{1}
&$\omega_1$:&
\dynkin[edge length=0.75cm,
labels*={,\frac{1}{2}}
]%
A[1]{1}
\\ \\
$\AAA_\ell$\ ($\ell\ge2$)
&\dynkin[edge length=0.75cm,
labels={,1,1,1,1},
labels*={1}
]%
A[1]{}
&$\omega_1$:&
\dynkin[edge length=0.75cm,
labels={,\frac{\ell}{\ell+1},\frac{\ell-1}{\ell+1},\frac{2}{\ell+1},\frac{1}{\ell+1}},
labels*={}
]%
A[1]{}
\\ \\
$\BBB_\ell$\ ($\ell\ge3$)
&\dynkin[%
edge length=0.75cm,
arrow shape/.style={-{angle 90}},
labels*={,,2,2,2,2,2},
labels={1,1}
]
B[1]{}
&$\omega_\ell$:&
\dynkin[%
edge length=0.75cm,
arrow shape/.style={-{angle 90}},
labels*={,,\frac{2}{2},\frac{3}{2},\frac{\ell-2}{2},\frac{\ell-1}{2},\frac{\ell}{2}},
labels={,\frac{1}{2}}
]
B[1]{}
\\ \\
$\CCC_\ell$\ ($\ell\ge2$)
&\dynkin[%
edge length=0.75cm,
arrow shape/.style={-{angle 90}},
labels*={1,2,2,2,2,1}%
]
C[1]{}
&$\omega_1$:&
\dynkin[edge length=0.75cm,
arrow shape/.style={-{angle 90}},
labels*={,0,0,0,0,\frac{1}{2}}]%
C[1]{}
\\ \\
$\DDD_\ell$\ ($\ell\ge4$)
&\dynkin[%
edge length=0.75cm,
arrow shape/.style={-{angle 90}},
labels={1,1,,,,,1,1},
labels*={,,2,2,2,2,,}
]
D[1]{}
&\multicolumn{2}{@{}l@{}}{$\begin{array}{@{}l@{\qquad}c}
\omega_\ell:&
\dynkin[edge length=0.75cm,
arrow shape/.style={-{angle 90}},
labels={,\frac{1}2,,,,\frac{\ell-2}2,\frac{\ell-2}4,\frac{\ell}4},
labels*={,,\frac{2}2,\frac{3}2,\frac{\ell-3}2,,,},
label directions={,,,,,right,,}
]
D[1]{}
\\ \\
\omega_1:&
\dynkin[edge length=0.75cm,
arrow shape/.style={-{angle 90}},
labels={,0,,,, ,\frac12,\frac12},
labels*={,,0,0,0,0,,}
]
D[1]{}
\end{array}$}
\\ \\
$\EEE_6$
&\dynkin[%
upside down,
edge length=0.75cm,labels={1,1,2,2,,2,1}%
]
E[1]{6}
& $\omega_1$:&
\dynkin[edge length=0.75cm,
upside down,
labels={,\frac{1}{3},0,\frac{2}{3}, ,\frac{1}{3},\frac{2}{3}}]%
E[1]{6}
\\ \\
$\EEE_7$
&\dynkin[%
edge length=0.75cm,
backwards,
upside down,
labels={1,2,2,,,,2,1},
label directions={,,left,,,,,}%
]
E[1]{7}
& $\omega_7$:&
\dynkin[edge length=0.75cm,
backwards,
upside down,
labels={,0,\frac{1}{2},,,,0,\frac{1}{2}},
label directions={,,left,,,,,}]%
E[1]{7}
\\ \\
$\EEE_8$
&\dynkin[%
edge length=0.75cm,
backwards,
upside down,
labels={1,2,,,,,,,2}%
]
E[1]{8} & &
\\ \\
$\FFF_4$
&\dynkin[%
edge length=0.75cm,
arrow shape/.style={-{angle 90}},
backwards,
labels*={1,2,,,2}%
]
F[1]{4} & &
\\ \\
$\GGG_2$
&\dynkin[%
edge length=0.75cm,
arrow shape/.style={-{angle 90}},
backwards,
labels*={1,2,}
]
G[1]{2} & &
\end{longtable}
\end{landscape}

\noindent
is a fundamental domain for the affine Weyl group $\Wtil^\ssc$ acting on $\ttl(\R)$.
\end{subsec}


\begin{subsec}\label{ss:C-action}
The action of $C=P^\vee/Q^\vee\simeq\Wtil^{\ad}/\Wtil^\ssc$ on $\Delta\simeq\ttl(\R)/\Wtil^\ssc$
is given by permutations of barycentric coordinates corresponding to a subgroup
of the automorphism group of the extended Dynkin diagram
acting simply transitively on the set of vertices $\alpha_i$ with $m_i=1$.
This action is described in \cite[Section VI.2.3, Proposition 6]{Bourbaki}.
Namely, the nonzero cosets of $C$ are represented
by the fundamental coweights $\omega_i^{\vee}$ such that $\ii\omega_i^{\vee}\in\Delta$, that is, $m_i=1$.
Let $w_0$, resp.\ $w_i$, denote the longest element in $W$, resp.\
in the Weyl group $W_i$ of the root subsystem $\Rt_i$ generated by $\Sm\smallsetminus\{\alpha_i\}$.
Then the transformation $x\mapsto w_i\hs w_0\hs x+\ii\omega_i^{\vee}$
preserves $\Delta$ whenever $m_i=1$ and gives the action of the respective coset
$[\omega_i^\vee]\in C$ on $\Delta$.

We describe the action  of $C$ on $\Dtil$ explicitly case by case,
using \cite[Plates I-IX, assertion (XII)]{Bourbaki}.
If $G$ is of one of the types $\EEE_8$, $\FFF_4$, $\GGG_2$, then $C=0$.
If $G$ is of one of the types $\AAA_1$, $\BBB_\ell$ ($\ell\ge 3$), $\CCC_\ell$ ($\ell\ge 2$), $\EEE_7$,
then $C\simeq\Z/2\Z$,
and the nontrivial element of $C$ acts on $\Dtil$ by the only nontrivial automorphism of $\Dtil$.

It remains to consider the cases $\AAA_\ell$ ($\ell\ge 2$), $\DDD_\ell$, and $\EEE_6$.
The action of
$C$ on $\Dtil$ is uniquely determined by restriction to
the set of vertices $\alpha_j$ of $\Dtil$ with $m_j=1$.
These are the images of $\alpha_0$ under the automorphism group of $\Dtil$.

Let $D$ be of type $\AAA_\ell$, $\ell\ge2$.
The generator $[\omega_1^\vee]$ of $C\simeq\Z/(\ell+1)\Z$ acts on $\Dtil$ as the
cyclic permutation $0\mapsto 1\mapsto\dots\mapsto\ell-1\mapsto\ell\mapsto 0$:

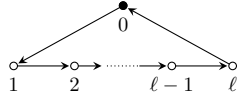
\tikzset{/Dynkin diagram,mark=o,affine mark=*,
edge length=0.75cm,arrow shape/.style={-{angle 90}}}

\[
\dynkin[%
edge length=0.8cm,
edge/.style={-{stealth[sep=2pt]}},
labels={,1,2,\ell-1,\ell},
labels*={0}]
A[1]{}
\]
\medskip

Let $D$ be of type $\DDD_\ell$, $\ell\ge 4$ is even.
We have $C\simeq\Z/2\Z\oplus\Z/2\Z$,
and  the classes $[\omega_1^\vee]$ and $[\omega_{\ell-1}^\vee]$ are generators of $C$.
These generators act on $\Dtil$ as follows: $[\omega_1^\vee]$ acts as $0\leftrightarrow1$, $\ell-1\leftrightarrow\ell$,
and $[\omega_{\ell-1}^\vee]$ acts as $0\leftrightarrow\ell-1$, $1\leftrightarrow\ell$:
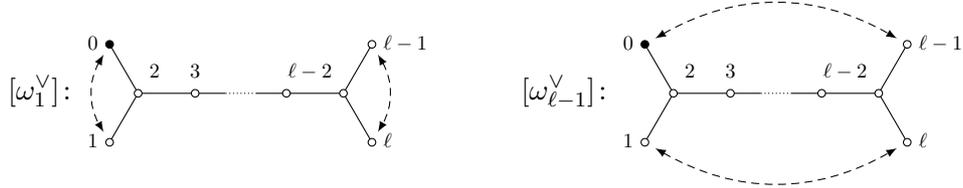
\[
[\omega_1^\vee]\colon
\dynkin[%
edge length=0.75cm,
labels={0,1,,,,,\ell-1,\ell},
labels*={,,2,3,,\ell-2,,},
involution/.style={latex-latex,densely dashed},
involutions={01;*67}]
D[1]{}
\qquad\quad
[\omega_{\ell-1}^\vee]\colon
\dynkin[%
edge length=0.75cm,
labels={0,1,,,,,\ell-1,\ell},
labels*={,,2,3,,\ell-2,,},
involution/.style={latex-latex,densely dashed},
involutions={*06;17}]
D[1]{}
\]
\medskip

Let $D$ be of type $\DDD_\ell$, $\ell\ge 5$ is odd.
We have $C\simeq\Z/4\Z$, and the element $[\omega_{\ell-1}^\vee]$
is a generator of $C$.
This generator acts on $\Dtil$ as  the 4-cycle $0\mapsto\ell-1\mapsto1\mapsto\ell\mapsto0$:

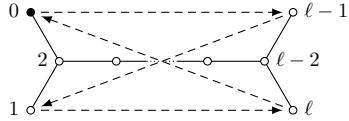
\[
 \begin{dynkinDiagram}[edge length=0.75cm,
labels={0,1,2,,,\ell-2,\ell-1,\ell},
labels*={,,,,,,,},
label directions={,,left,,,right,,}]
{D}[1]{}
\draw[-latex,densely dashed] (root 0) to (root 6);
\draw[-latex,densely dashed] (root 1) to (root 7);
\draw[-latex,densely dashed] (root 7) to (root 0);
\draw[-latex,densely dashed] (root 6) to (root 1);
\end{dynkinDiagram}
\]

\bigskip

Let $D$ be of type $\EEE_6$.
The generator $[\omega_1^{\vee}]$ of $C\simeq\Z/3\Z$  acts as the 3-cycle $0\mapsto1\mapsto5\mapsto0$:

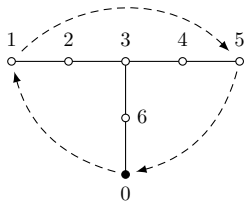
\[
\dynkin[%
edge length=0.75cm,
upside down,
labels={,1,6,2,3,4,5},%
labels*={0},%
involution/.style={-latex,densely dashed},
involutions={[in=-135,out=-45,relative]16;60;01}]
E[1]{6}
\]

\bigskip

We state results of this section as a proposition.
\end{subsec}

\begin{proposition}\label{p:Delta-ttt-inner}
With the assumptions and notation of this section,
 the inclusion map $\Delta\into\ttl(\R)$
 induces a $C$-equivariant bijective correspondence
 between  $\Delta$ and the set of orbits of\/ $\Wtil^\ssc$ in $\ttl(\R)$,
 where $C$ acts on
 $\ttl(\R)/\Wtil^\ssc$ via the isomorphism $C\simeq \Wtil^\ad/\Wtil^\ssc$,
 and  on $\Delta$ by permutations of barycentric coordinates as described above.
 \end{proposition}

\section{Case of an outer form of a simple compact group}
\label{s:outer}

\begin{subsec}
In this section, $\GG$ is a simply connected, simple, compact $\R$-group,
and $\tau\in\Aut\BRD(G)$ {\em is of order} 2.
Then $G$ is of one of the types $\AAA_m$ ($m\ge 2)$,
$\DDD_m$ ($m\ge3)$, or $\EEE_6$\hs.
We describe the quotient set $\ttl_0(\R)/\Wotil^\ssc$, the group $C_0=\widetilde{P}_0^\vee/\widetilde{Q}_0^\vee$
and the action of $C_0$ on $\ttl_0(\R)/\Wotil^\ssc$\hs.
We use the notation from Section~\ref{s:lattices}.

The restricted root systems $\Rtbar$ are tabulated
in Gorbatsevich, Onishchik and Vinberg \cite[Section 3.3.9, p.\,119]{OV2}.
If $\Rt$ is of type $\AAA_{2l}$  $(l\ge 1)$, then $\Rtbar$ is of type $\BBB\CCC_l$;
if $\Rt$ is of type $\AAA_{2l-1}$ $(l\ge3)$, then $\Rtbar$ is of type $\CCC_l$;
if $\Rt$ is of type $\DDD_{l+1}$ $(l\ge2)$, then $\Rtbar$ is of type $\BBB_l$;
if $\Rt$ is of type $\EEE_{6}$, then $\Rtbar$ is of type $\FFF_4$.

Recall that $\gl^{\pm\ttau}$ denote the $(\pm1)$-eigenspaces for $\ttau$ in $\gl$.
\end{subsec}

\begin{lemma}[{\cite[Section 3.3.9]{OV2}}]
$\gl^{\ttau}$ is a simple Lie subalgebra of $\gl$
and the adjoint representation of $\gl^{\ttau}$ in $\gl^{-\ttau }$ is irreducible.
\end{lemma}

\begin{proof}
In \cite{OV2} the statement is verified by explicit calculations in the four above cases.
Here is a conceptual argument.

We compute the pairing between restricted simple roots and coroots,
with the notation of Lemma~\ref{l:Q0P0}:
$$
\langle\bar\alpha_i,\bar\alpha_j^\vee\rangle=\langle\alpha_i,\alpha_j^\vee\rangle,\qquad
\langle\bar\alpha_i,\bar\alpha_k^\vee\rangle=\langle\alpha_i,\alpha_k^\vee+\tau(\alpha_k^\vee)\rangle
=\langle\alpha_i,\alpha_k^\vee\rangle+\langle\alpha_i,\tau(\alpha_k^\vee)\rangle.
$$
It follows from this computation that simple roots $\alpha_p,\alpha_q$ of $\Rt$
which are linked in the Dynkin diagram $\D$ of $\Rt$,
that is, have negative Cartan number, restrict to simple roots $\bar\alpha_p,\bar\alpha_q$ of $\Rtbar$
which are linked in the Dynkin diagram $\Dbar$ of $\Rtbar$. Hence $\Dbar$ is connected.
Since $\Dbar$ is also the Dynkin diagram of $\gl^{\ttau}$ and the number of its vertices
is $\ell=\dim\tl_0$, the latter Lie algebra is simple.
Now the irreducibility of $\gl^{-\ttau}$ follows by an argument
in \cite[Section 3.3.11]{OV2}, right above Lemma~3.17 of loc.~cit.
\end{proof}

\begin{subsec}\label{ss:ext-outer}
Let $\bar\alpha_0$ denote the lowest weight of $T_0$ in $\gl^{-\ttau }$.
Then $\Smtil=\{\bar\alpha_0,\bar\alpha_1,\dots,\bar\alpha_{\ell}\}$ is an admissible system of roots
in the sense that the Cartan numbers
of all pairs $\bar\alpha_p,\bar\alpha_q$ with $p\ne q$
are non-positive.
It is encoded by a twisted affine Dynkin diagram
$\Dtil=\Dtil(G,T,B,\ttau)=\Dtil(\Rt,\Sm,\tau)$ in the usual way;
see \cite[Section 3.1.7]{OV2}.
There is a unique linear dependence
\begin{equation}\label{e:mj-outer}
m_0\bar\alpha_0+m_1\bar\alpha_1+\dots+m_{\ell}\bar\alpha_{\ell}=0,
\end{equation}
where $m_i$ are positive even integers and $m_0=2$.
These coefficients $m_i$ are tabulated in \cite[Section 3.3.9]{OV2};
see also Table~\ref{t:tab-t} below.
We write only the coefficients $m_i=2$.
For $^2\hm \AAA_{2\ell-1}$ and $^2\hm \DDD_{\ell+1}$, we write also the coefficients $c_i$ (modulo $\Z$)
in the decomposition of a representative
of the generator of $P_0/Q_0\simeq\Z/2\Z$
into a linear combination of simple restricted roots.
We write these coefficients $c_i$ only for simple roots with $m_i=2$.
The vertex corresponding to $\alpha_0$ is painted in black.

Each $x\in\ttl_0(\R)$ has \emph{barycentric coordinates} $x_0,x_1,\dots,x_{\ell}$ defined by
\[\bar\alpha_i(x)=\ii{x_i}\quad (i=1,\dots,\ell), \quad \bar\alpha_0(x)=\ii(x_0-\tfrac12).\]
 They satisfy the identity coming from \eqref{e:mj-outer}:
$$m_0x_0+m_1x_1+\dots+m_{\ell}x_{\ell}=1.$$
\end{subsec}

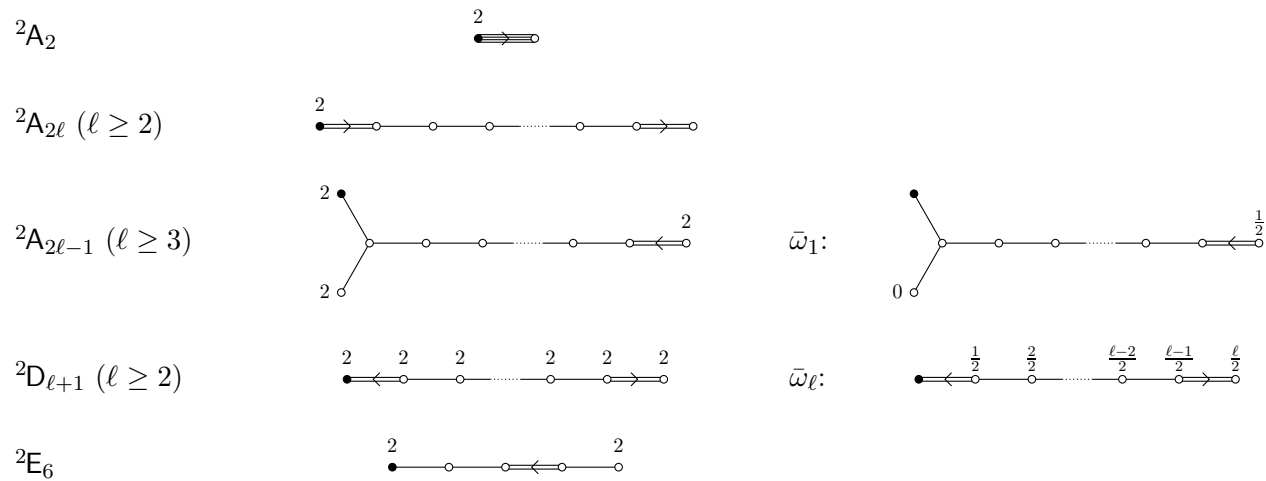
\tikzset{/Dynkin diagram,mark=o,affine mark=*,
edge length=0.75cm,arrow shape/.style={-{angle 90}}}

\begin{landscape}
\begin{longtable}[c]{l@{\qquad\qquad}c@{\qquad\quad}l@{\qquad}c}
\caption{Coefficients $m_i$ and $c_i$ on twisted affine Dynkin diagrams}
\label{t:tab-t}
\\[7ex]
\endfirsthead
\caption{(continued)}
\\[7ex]
\endhead
\endfoot
$^2\hm \AAA_2$
&\dynkin[%
edge length=0.75cm,
arrow shape/.style={-{angle 90}},
reverse arrows,
labels*={2,}
]
A[2]{2} &&
\\ \\
$^2\hm \AAA_{2\ell}$\ ($\ell\ge2$)
&\dynkin[%
edge length=0.75cm,
arrow shape/.style={-{angle 90}},
reverse arrows,
labels*={2}
]
A[2]{even} &&
\\ \\
$^2\hm \AAA_{2\ell-1}$\ ($\ell\ge3$)
&\dynkin[%
edge length=0.75cm,
arrow shape/.style={-{angle 90}},
labels={2,2,},
labels*={ , ,,,,,,2}
]
A[2]{odd}
& $\bar\omega_1$: &
\dynkin[%
edge length=0.75cm,
arrow shape/.style={-{angle 90}},
labels={,0,},
labels*={,,,,,,,\frac{1}{2}}
]
A[2]{odd}
\\ \\
$^2\hm \DDD_{\ell+1}$\ $(\ell\ge2)$
&\dynkin[%
edge length=0.75cm,
arrow shape/.style={-{angle 90}},
labels*={2,2,2,2,2,2}
]
D[2]{}
& $\bar\omega_\ell$: &
\dynkin[%
edge length=0.75cm,
arrow shape/.style={-{angle 90}},
labels*={,\frac{1}{2},\frac{2}{2},\frac{\ell-2}{2},\frac{\ell-1}{2},\frac{\ell}{2}}
]
D[2]{}
\\ \\
$^2\hm \EEE_6$
&\dynkin[%
edge length=0.75cm,
arrow shape/.style={-{angle 90}},
labels*={2,,,,2}
]
E[2]{6} &&
\end{longtable}
\end{landscape}

\begin{subsec}\label{ss:simplex-outer}
The simplex
\begin{equation*}\label{e:simplex}
\Delta=\{x\in\ttl_0(\R)\mid x_0,x_1,\dots,x_{\ell}\ge0\}
\end{equation*}
is a fundamental domain for the $\Wotil^\ssc$-action on $\ttl_0(\R)$;
see \cite[Section 3.3.10, Prop.\,3.15(3)]{OV2}.
The  group $C_0=\widetilde{P}_0^{\vee}/\widetilde{Q}_0^{\vee}\simeq\Wotil^{\ad}/\Wotil^\ssc$
acts on $\Delta\simeq\ttl_0(\R)/\Wotil^\ssc$
by permuting the barycentric coordinates in accordance with automorphisms of $\Dtil$.
A general description of this action is very similar to the one given in Subsection~\ref{ss:C-action}
for the action of $C$ on the fundamental simplex of $\Wtil^{\ssc}$.
Namely, the nonzero cosets of $C_0$ are represented
by the vectors $\nu\in\wt P_0^{\vee}$ such that $\ii\nu\in\Delta$,
which are exactly $\nu=\bar\omega_k^\vee/2$ such that $\alpha_k$ is moved by $\tau$ and $m_k=2$.
 If $w_0$, resp.\ $w_k$,
denote the longest element in $W_0$, resp.\ in the Weyl group $W_k$
of the root subsystem $\Rtbar_k$ generated by $\Smbar\smallsetminus\{\bar\alpha_k\}$,
then the transformation $x\mapsto w_k\hs w_0\hs x+\ii\bar\omega_k^{\vee}/2$
preserves $\Delta$ and gives the action of the respective coset
$[\bar\omega_k^\vee/2]\in C_0$ on~$\Delta$. The proof given in \cite[Section VI.2.3, Proposition 6]{Bourbaki}
for the non-twisted case works in the twisted case with obvious modifications.

Specifically, if $G$ is of type $\AAA_{2l}$ $(\ell\ge 1)$ or $\EEE_6$, then $C_0=\{0\}$.
If $G$ is of type $\AAA_{2l-1}$ $(\ell\ge3)$ or $\DDD_{l+1}$ $(\ell\ge 2)$,
then $C_0\simeq\Z/2\Z$ is generated by $[\bar\omega_k^\vee/2]$,
where $\bar\alpha_k$ is the unique vertex of $\Dbar$
which is interchanged with $\bar\alpha_0$ under the unique nontrivial (involutive) automorphism of $\Dtil$.
The coset $[\bar\omega_k^\vee/2]$ acts by this automorphism.

We state results of this section as a proposition.
\end{subsec}

\begin{proposition}\label{p:Delta-ttt-outer}
With the assumptions and notation of this section,
the inclusion $\Delta\into \ttl_0(\R)$
induces a $C_0$-equivariant bijective correspondence
between  $\Delta$ and the orbit set of\/ $\widetilde{W}^\ssc _0$ in $\ttl_0(\R)$,
where $C_0$ acts on
$\ttl_0(\R)/\Wotil^\ssc $ via the isomorphism $C_0\simeq \Wotil^{\ad}/\Wotil^\ssc $,
and on $\Delta$ by permutations of coordinates as described above.
\end{proposition}

\section{Case of an $\R$-simple non-absolutely simple group}
\label{s:nas}

\begin{subsec}
Now suppose that $\GG$ is simply connected and
 not simple, but $\thG$ is $\R$-simple.
This means that $\GG\simeq\GG'\times_\R\GG'$,
where $\GG'$ is a compact, simply connected, {\em simple} $\R$-group,
and $\theta$ swaps the two simple factors of $\GG$,
that is, $\theta(g_1,g_2)=(g_2,g_1)$ for $g_1,g_2\in G'$.

We choose a maximal torus $\TT=\TT'\times\TT'$ in $\GG$, where $\TT'$ is a maximal torus in~$\GG'$.
We also choose a Borel subgroup $B=B'\times B'$ in $G$ containing $T$,
where $B'$ is a Borel subgroup in $G'$ containing $T'$.

Then $\theta$ preserves $T$ and $B$.
Furthermore, a pinning of $(G',T',B')$ gives rise to a ``doubled'' pinning of $(G,T,B)$ preserved by $\theta$.
Hence $\theta=\ttau$.

The torus $\TT_0=\TT^\ttau$ given by
\[ T_0=\{(t,t)\,\mid\, t\in T'\}\subset T\]
is a maximal torus in the diagonal subgroup $\GG^\ttau\subset\GG=\GG'\times_\R\GG'$.
The torus $\TT_1$ is given by
\[ T_1=\{(t,t^{-1})\,\mid\, t\in T'\}\subset T.\]
We shall identify $G^\ttau$ with $G'$ and $T_0$ with $T'$, both embedded diagonally.
The representation of $G^\ttau$ in $\gl^{-\ttau}=\{(\xi,-\xi)\mid\xi\in\gl'\}$
is equivalent to the adjoint representation of $\gl'$.
\end{subsec}

\begin{subsec}\label{ss:res-nas}
Let $P^{\hs\prime}$, $Q^{\hs\prime}$, $P^{\hs\prime\hs\vee}$, $Q^{\hs\prime\hs\vee}$
denote the lattices of weights, roots, coweights, and coroots of $G'$ with respect to $T'$, correspondingly.
With the notation of Section~\ref{s:lattices}, we have decompositions:
\begin{align*}
P        &= P^{\hs\prime}\oplus P^{\hs\prime},               &   Q &= Q^{\hs\prime}\oplus Q^{\hs\prime}, \\
P^{\vee} &= P^{\hs\prime\hs\vee}\oplus P^{\hs\prime\hs\vee}, &
Q^{\vee} &= Q^{\hs\prime\hs\vee}\oplus Q^{\hs\prime\hs\vee}.
\end{align*}
The restriction of characters from $T$ to $T_0=T'$ is the map
$$\chi=(\chi_1,\chi_2)\,\longmapsto\,\bar\chi=\chi_1+\chi_2.$$
The restricted root system $\Rtbar$ is just the root system of $G'$ with respect to $T'$.
The restricted simple roots $\bar\alpha_1,\dots,\bar\alpha_\ell$
are the simple roots of $G'$ and $\bar\alpha_0$ is the lowest root of $G'$.
The lattices $P_0$ and $Q_0$ are identified with $P'$ and $Q'$, respectively.
The lattices $P_0^{\vee}=P^{\hs\prime\hs\vee}$ and $Q_0^{\vee}=Q^{\hs\prime\hs\vee}$
are diagonally embedded in $P^{\vee}$ and $Q^{\vee}$, respectively.
The projection
from $\tl$ to $\tl_0$ is the map
$$x=(x_1,x_2)\,\longmapsto\,\bar x=\half(x_1+x_2).$$
Under these identifications, we have
$$
\widetilde{P}_0^{\vee}=\half P^{\hs\prime\hs\vee}\quad\text{and}\quad
\widetilde{Q}_0^{\vee}=\half Q^{\hs\prime\hs\vee}.
$$
\end{subsec}

\begin{subsec}\label{ss:barycentric-nas}
The group $W_0$ is isomorphic to the Weyl group of $(G',T')$ and
$\Wotil^\ssc$ is isomorphic to the affine Weyl group of $G'$.
The simplex $\Delta=\half\Delta'$  is a  fundamental domain in $\ttl_0(\R)=\ttl'(\R)$ for $\Wotil^\ssc$,
where  $\Delta'$ is the fundamental simplex of the affine Weyl group of $G'$, as in Section \ref{s:inner}.

The barycentric coordinates on $\ttl_0(\R)=\ttl'(\R)$ are defined by
\begin{equation*}
\alpha_i(x)=\ii x_i\quad\text{for }i=1,\dots,\ell,\quad\alpha_0(x)=\ii(x_0-\half),
\end{equation*}
which coincides with definition \eqref{e:barycentric} in Section \ref{s:inner} except for
$\half$ instead of
$1$ in the last equality, so that
$x_0,x_1,\dots,x_\ell\ge0$ are still
the defining inequalities for~$\Delta$.
We set $m_i=2m'_i$ $(i=0,1,\dots,\ell)$, where $m_i'$ are the coefficients for $G'$ as in \eqref{e:sum-mj}. Then
\begin{equation*}
m_0x_0+m_1x_1+\dots+m_{\ell}x_{\ell}=1.
\end{equation*}
The group $C_0\simeq P^{\hs\prime\hs\vee}/Q^{\hs\prime\hs\vee}$
acts on $\Delta$ by permuting the barycentric coordinates in accordance
with automorphisms of the extended Dynkin diagram $\Dtil=\Dtil(G',T',B')$,
as in Section \ref{s:inner}.

We state results of this section as a proposition.
\end{subsec}

\begin{proposition}\label{p:Delta-ttt-nas}
With the assumptions and notation of this section,
the inclusion $\Delta\into \ttl_0(\R)\simeq\ttl'(\R)$
induces a $C_0$-equivariant bijective correspondence
between  $\Delta$ and the orbit set of $\Wotil^\ssc$ in $\ttl_0(\R)$,
where $C_0\simeq P^{\hs\prime\hs\vee}/Q^{\hs\prime\hs\vee}$ acts on
$\ttl_0(\R)/\Wotil^\ssc $ via the isomorphism $C_0\simeq \Wotil^\ad/\Wotil^\ssc$,
and on $\Delta=\half\Delta'$ by permutations of coordinates as described in Section \ref{s:inner}.
\end{proposition}

\section{Square roots of a central element}
\label{s:sqroots}

\begin{subsec}
We retain Notation \ref{n:not2} assuming additionally that
$\GG=(G,\sigma_c)$ is a compact \emph{semisimple} $\R$-group, not necessarily simply connected.
We write $\GG^\ssc$ for the universal cover of $\GG$.
The involutions $\sigma_c$ and $\theta$
lift to commuting involutions of~$G^\ssc$, which we denote by the same letters. Consider
the twisted group $\thG$ and its universal cover $\thG^\ssc$.
We have a decomposition into $\R$-simple factors
\begin{equation*}
\thG^\ssc = \thG^{(1)} \times_\R\dots\times_\R \thG^{(r)}.
\end{equation*}
It corresponds to a decomposition
\begin{equation*}
 \GG^\ssc =  \GG^{(1)} \times_\R\dots\times_\R  \GG^{(r)},
\end{equation*}
in which each pair $(\GG^{(k)},\theta^{(k)})$ is indecomposable,
where $\theta^{(k)}$ is the restriction of $\theta$ to $\GG^{(k)}$.

We consider
the spaces $\tl_0$, $\tl_0^*$, $\ttl_0(\R)$, and various lattices in them, see Section~\ref{s:lattices};
each of these objects decomposes into a direct sum of the corresponding objects for $G^{(k)}$,
which we denote by adding the superscript ${}^{(k)}$.
In particular,
\[\ttl_0(\R)=\ttl_0^{(1)}(\R)\oplus\dots \oplus \ttl_0^{(r)}(\R),\]
Similarly, the groups $W_0$ and $\Wotil^\ssc$ decompose
into the direct products of the corresponding groups for $G^{(k)}$.
We obtain a decomposition of the orbit set
\begin{equation*}
\ttl_0(\R)/\Wotil^\ssc=\ttl^{(1)}_0(\R)/\Wotil^{(1)\hshs\ssc}\times\dots\times\ttl^{(r)}_0(\R)/\Wotil^{(r)\hshs\ssc}
\end{equation*}
compatible with the decomposition of the group
\begin{equation*}
C_0=C^{(1)}_0\oplus\dots\oplus C^{(r)}_0,
\end{equation*}
where the group $C_0=\wt P^{\vee}_0\hmm/\wt Q^{\vee}_0$ acts on the orbit set $\ttl_0(\R)/\Wotil^\ssc$,
and each group $C^{(k)}_0=\wt P^{(k)\vee}_0\hmm\hmm/\wt Q^{(k)\vee}_0$
acts on the respective orbit set $\ttl_0^{(k)}(\R)/\Wotil^{(k)\hshs\ssc}$.
\end{subsec}

\begin{subsec}
There are three cases:
\begin{enumerate}
\renewcommand{\theenumi}{\alph{enumi}}
  \item\label{i:inner-k} $\GG^{(k)}$ is simple and $\theta^{(k)}$ is an {\em inner} involution of $\GG$;
  \item\label{i:outer-k} $\GG^{(k)}$ is simple and $\theta^{(k)}$ is an {\em outer} involution of $\GG$;
  \item\label{i:swap-k} $\GG^{(k)}=\GG^{(k)\hs\prime}\times_\R \GG^{(k)\hs\prime\prime}$ and
  $\theta^{(k)}$ swaps the isomorphic simple factors $\GG^{(k)\hs\prime}$ and $\GG^{(k)\hs\prime\prime}$.
\end{enumerate}
They were examined in detail in Sections \ref{s:inner}, \ref{s:outer}, and \ref{s:nas}, respectively.

Let $\Smtil=\Smtil^{(1)}\sqcup\dots\sqcup\Smtil^{(r)}$ denote the union
of the extended sets of simple restricted roots of $(G^{(k)},\ttau^{(k)})$,
where $\tau^{(k)}=\theta^{(k)}_\brd\in\Aut\BRD(G^{(k)})$;
see Subsections \ref{ss:ext-inner}, \ref{ss:ext-outer}, and \ref{ss:res-nas}.
We identify $\Smtil$  with the set of vertices of the diagram
\[\Dtil=\Dtil(G,T,B,\ttau)=\Dtil^{(1)}\sqcup\dots\sqcup \Dtil^{(r)},\]
where $\Dtil^{(k)}=\Dtil(G^{(k)},T^{(k)},B^{(k)},\ttau^{(k)})$
is an affine Dynkin diagram (twisted or not), as defined in Sections \ref{s:inner}--\ref{s:nas}.

Each point $x=(x^{(1)},\dots,x^{(r)})\in\ttl_0(\R)$ has barycentric coordinates $(x_\beta)_{\beta\in\Smtil}$\hssh,
which are defined separately for each component $x^{(k)}\in\ttl_0^{(k)}(\R)$
(see Subsections \ref{ss:barycentric-inner}, \ref{ss:ext-outer}, and \ref{ss:barycentric-nas}).
They satisfy
\begin{equation}\label{e:sum-bcc-k}
\sum_{\beta\in\Smtil^{(k)}}\!\!\! m_\beta x_\beta=1\quad\text{for each }k=1,\dots,r,
\end{equation}
where $m_\beta$ are the coefficients of the linear dependence of $\Smtil^{(k)}$
normalized as in Subsections \ref{ss:barycentric-inner}, \ref{ss:ext-outer}, and \ref{ss:barycentric-nas}.

The product
\[\Delta=\Delta^{(1)}\times\dots\times\Delta^{(r)}\]
is a fundamental domain for the group $\Wotil^\ssc$ acting on $\ttl_0(\R)$,
where $\Delta^{(k)}$ are the fundamental simplices
for the respective affine Weyl groups $\Wotil^{(k)\hshs\ssc}$ acting on $\ttl_0^{(k)}(\R)$;
see Subsections \ref{ss:barycentric-inner}, \ref{ss:simplex-outer}, and \ref{ss:barycentric-nas}.
The defining inequalities for $\Delta$ are $x_{\beta}\ge0$ over all $\beta\in\Smtil$.
The quotient group ${F}_0=\Wotil/\Wotil^\ssc\simeq\widetilde{X}_0^{\vee}/\widetilde{Q}_0^{\vee}$
acts on $\Delta\simeq\ttl_0(\R)/\Wotil^\ssc$
as a subgroup of $C_0$, where $C_0$ acts by permuting the barycentric coordinates in accordance with automorphisms of $\Dtil$.
\end{subsec}

\begin{proposition}\label{p:Delta-ttt-ss}
The inclusion $\Delta\into \ttl_0(\R)$
and the map $\Exphat\colon\ttl_0(\R) \rightarrow {\TT_0(\R)\cdot\htau}$ of  \eqref{e:ttt0-e0} induce
bijective correspondences between the orbit sets for the actions of ${F}_0$ in $\Delta$,
of\/ $\Wotil$ in $\ttl_0(\R)$, and of\/ $\widehat{W}_0$ in $\TT_0(\R)\cdot\htau$.
\end{proposition}

\begin{proof}
Propositions \ref{p:Delta-ttt-inner}, \ref{p:Delta-ttt-outer}, and \ref{p:Delta-ttt-nas} show
that for each $k=1,\dots,r$,  the $C_0^{(k)}$-equivariant  map  $\Delta^{(k)}\longrightarrow\ttl_0^{(k)}(\R)/\Wotil^{(k)\hshs\ssc}$  is bijective.
Now it follows from the product structure  of $\Delta$, $\ttl_0(\R)$, $\Wotil^{\ssc}$, and $C_0$,
that the $C_0$-equivariant map $$\Delta\longrightarrow\ttl_0(\R)/\Wotil^{\ssc}$$ is bijective.
Since $F_0\subseteq C_0$, this bijective map is $F_0$-equivariant and induces a bijection of the orbit sets
\[\Delta/F_0\,\longisoto\, (\ttl_0(\R)/\Wotil^\ssc)/F_0=\ttl_0(\R)/\Wotil.\]
Finally, Lemma~\ref{l:Wtil-What} gives a bijection  $\ttl_0(\R)/\Wotil\isoto (\TT_0(\R)\cdot\htau)/\widehat{W}_0$\hs.
\end{proof}

Our aim is to describe a smaller orbit set for the action of $\Wohat$ on $(T_0\cdot\htau)^z_2\subset\TT_0(\R)\cdot\htau$.
We do it in terms of a combinatorial notion which we introduce now.

\begin{definition}
A {\em Kac labeling} of $\Dtil$ is a family of nonnegative integer  numerical labels
$\pp=(p_\beta)_{\beta\in\Dtil}$
at the vertices $\beta\in\Smtil$ of $\Dtil$ satisfying
\begin{equation}\label{e:vk-p-0}
 \sum_{\beta\in\Smtil^{(k)}}\!\!\! m_\beta p_\beta=2\quad\text{for each }k=1,\dots,r.
\end{equation}
Labelings of this kind were used by Kac \cite{Kac} in classification
of automorphisms of finite order (specifically, involutions) of semisimple Lie algebras.
We denote the set of Kac labelings of $\Dtil$ by $\Km(\Dtil)$.
\end{definition}

Note that a Kac labeling $\pp$ of $\Dtil=\Dtil^{(1)}\sqcup\dots\sqcup \Dtil^{(r)}$ is the same as
a family $(\pp^{(1)},\dots,\pp^{(r)})$, where each $\pp^\vvk$ is a Kac labeling of $\Dtil^\vvk$.

\begin{subsec}
To any $x\in\ttl_0(\R)$, we assign a family $\pp=\pp(x)=(p_\beta)_{\beta\in\Smtil}$
of \emph{real} numbers $p_\beta=2x_\beta$, which satisfy \eqref{e:vk-p-0} by \eqref{e:sum-bcc-k}.
This correspondence identifies $\ttl_0(\R)$ with the subspace of $\R^{\Smtil}$
defined by equations \eqref{e:vk-p-0}.
We denote the inverse correspondence as
$$\pp\,\mapsto\, x=x(\pp)=\frac\ii2\sum_{\beta\in\Smbar}p_\beta\bar\omega^\vee_\beta,$$
where $\bar\omega^\vee_\beta=\bar\omega^\vee_i$ is the fundamental coweight corresponding
to a simple restricted root $\beta=\bar\alpha_i$, with the notation of Lemma~\ref{l:Q0P0}.
We have $x\in\Delta$ if and only if $p_\beta\ge0$ for all $\beta\in\Smtil$.

For any character $\lambda\in X_0$, $\lambda=\sum_{\beta\in\Smbar}\hs c_\beta\hs\beta$, $c_\beta\in\Q$, put
$$
\langle\lambda,\pp\rangle=\sum_{\beta\in\Smbar}c_\beta\hs p_\beta=2\lambda(x(\pp))/\ii.
$$
Note that $\langle\lambda,\pp\rangle\in\Q$ whenever $\pp\in\Km(\Dtil)$,
and $\langle\lambda,\pp\rangle\in\Z$ if furthermore $\lambda\in Q_0$.

Recall that $\theta=\inn(\tto)\circ\ttau$, where $\tto\in T_0$ and $z=\tto^2\in Z(G)\cap T_0$. We write
\begin{equation}\label{e:l-z}
z=\exp 2\pi\ii\zeta,\quad \text{where } \zeta\in\tl_0\hs.
\end{equation}
For a character $\lambda\in X_0\subset\tl_0^*$, we have
\begin{equation}\label{e:z}
\lambda(z)=\lambda(\exp 2\pi\ii\zeta)=\exp \lambda(2\pi\ii\zeta)=\exp 2\pi\ii\lambda(\zeta).
\end{equation}
Since $z$ is an element of finite order in $T_0$, we see that $\lambda(z)$ is a root of unity.
Hence by \eqref{e:z} $\lambda(\zeta)\in\Q$ and, moreover, $\lambda(\zeta)\in\Z$ for $\lambda\in Q_0$
since $z\in Z(G)$ and therefore $\lambda(z)=1$ in this case.
It follows from \eqref{e:z} that the image of $\lambda(\zeta)$ in $\Q/\Z$
depends only on $z$ and not on the choice of $\zeta$, which is determined modulo~$X_0^\vee$.

The following theorem gives a combinatorial description
of the set $(T_0\cdot\htau)^z_2/\Wohat $ in terms of Kac labelings.
\end{subsec}

\begin{theorem}\label{t:p-z}
Let $\zeta\in\tl_0$ be as in \eqref{e:l-z}.
There is a  canonical bijection between the set $(T_0\cdot\htau)_2^z\hs/\Wohat$
and the set of ${F}_0$-orbits
in the set of Kac labelings $\pp\in\Km(\Dtil)$
satisfying
\begin{equation}\label{e:cond-z}
\langle\lambda,\pp\rangle\equiv\lambda(\zeta)\pmod\Z\quad\text{for all }[\lambda]\in X_0/Q_0\hs.
\end{equation}
The bijection sends the $F_0$-orbit of $\pp\in\Km(\Dtil)$ to the $\Wohat$-orbit of $\Exphat(x(\pp))\in T_0\cdot\htau$
with the notation of Section \ref{s:logarithm}.
\end{theorem}

\begin{proof}
Any element of $\TT_0(\R)\cdot\htau$ is of the form $t\cdot\htau=\Exphat(x)$ for some $x\in\ttl_0(\R)$,
so that $t=\Exp(x)\in T_0$ with the notation of Section \ref{s:lattices}.
We have $t\cdot\htau\in(T_0\cdot\htau)_2^z$ if and only if $t^2=z$.
Since $t^2=\Exp(2x)$ and $z=\Exp(\ii\zeta)$, the  condition $t^2=z$ is equivalent to
\begin{align*}
      && 2x&=\ii\zeta+\ii\nu && \text{for some }\nu\in X_0^\vee \\
\iff  && \ii\langle\lambda,\pp\rangle &= 2\lambda(x)=\ii\lambda(\zeta)+\ii\lambda(\nu)\equiv\ii\lambda(\zeta)\pmod{\ii\Z}
      && \text{for all }\lambda\in X_0\\
\iff  && \langle\lambda,\pp\rangle&\equiv\lambda(\zeta)\pmod\Z && \text{for all }\lambda\in X_0.
\end{align*}
In particular, taking $\lambda=\beta\in\Smbar\subset Q_0$, we get $p_\beta\equiv\beta(\zeta)\in\Z$.
Equality \eqref{e:vk-p-0} implies $p_\beta\in\Z$ for all $\beta\in\Smtil$, that is, $\pp\in\Z^\Smtil$.

The set of all $x\in\ttl_0(\R)$ such that $\pp=\pp(x)$ satisfies \eqref{e:cond-z} is $\Wotil$-stable,
being the preimage of a $\Wohat$-stable set $(T_0\cdot\htau)_2^z$ under $\Exphat$.
By Proposition~\ref{p:Delta-ttt-ss}, the $\Wotil$-orbits in this subset of $\ttl_0(\R)$ are the preimages
of the $\Wohat$-orbits in $(T_0\cdot\htau)_2^z$, and intersect $\Delta$ in $F_0$-orbits.
This yields a bijection between $(T_0\cdot\htau)_2^z\hs/\Wohat$
and the set of $F_0$-orbits of all $x\in\Delta$ with $\pp=\pp(x)$ satisfying \eqref{e:cond-z}.
The condition $x\in\Delta$ reads as $p_\beta\ge0$ for all $\beta\in\Smtil$.
Thus all $p_\beta$ are integer and non-negative, that is,  $\pp\in\Km(\Dtil)$, which completes the proof.
\end{proof}

\begin{subsec}\label{ss:can-R-str}
Theorem~\ref{t:p-z} implies, in particular, that the element $\tthat=\tto\cdot\htau\in(T_0\cdot\htau)_2^z$
is $\Wohat$-equivalent to an element of the form $\Exphat(x(\qq))$ for some Kac labeling $\qq\in\Km(\Dtil)$.
The action of $\Wohat$ on $(T_0\cdot\htau)_2^z$ comes from the conjugation action \eqref{e:Ntau-action} of $\NNt(\R)$
(see Lemma~\ref{l:N-M}\eqref{i:M} and Lemma~\ref{l:Ntau-action}).
If we replace $\tthat$ by an $\NNt(\R)$-conjugate, the involution $\theta$ will be replaced
by a conjugate one $\inn(g)\circ\theta\circ\inn(g)^{-1}$ with $g\in\NNt(\R)$.
This does not change the $\R$-group $\thG$ up to an isomorphism.
Thus we may assume without loss of generality that $\tthat=\Exphat(x(\qq))$ and $\tto=\Exp(x(\qq))$.

The corresponding semisimple $\R$-group $\thG$ is defined by the following data:
\begin{itemize}
  \item[\cc] a Dynkin diagram $\D$, which determines the Lie algebra $\gl$ and the simply connected group $G^\ssc$;
  \item[\cc] a lattice $X^\vee$ in between $P^\vee$ and $Q^\vee$ or,
  equivalently, a subgroup $F=X^\vee/Q^\vee\subseteq C=P^\vee/Q^\vee\simeq Z(G^\ssc)$,
  which determines the complex algebraic group $G=G^\ssc/F$;
  \item[\cc] a diagram automorphism $\tau\in\Aut\D$ preserving $F$,
  which determines the class of $\theta$ in $\Aut(G)/\Inn(G)$;
  \item[\cc] a Kac labeling $\qq\in\Km(\Dtil)$, which determines $\tto=\Exp(x(\qq))$
  and, in turn, the involution $\theta=\inn(\tto)\circ\ttau$
  and the real structure $\sigma=\theta\circ\sigma_c$ on $G$.
\end{itemize}
We write
\begin{equation*}\label{e:G(D,tau,qq,F)}
\thG=\GG(D,F,\tau,\qq).
\end{equation*}
The above discussion shows that every semisimple $\R$-group is isomorphic to a group of the form
$\GG(D,F,\tau,\qq)$ for some $D$, $F$, $\tau$, and $\qq$.
\end{subsec}

\section{Main theorem}
\label{s:mt-ss}

\begin{subsec}
We are now ready to state the main result of this article.  We retain Notation~\ref{n:not2}
and the notation from Sections \ref{s:lattices} and \ref{s:sqroots}.

Consider a semisimple $\R$-group $\GG(\D,F,\tau,\qq)$, as defined in Subsection~\ref{ss:can-R-str},
where $D$ is a Dynkin diagram, $F$ is a subgroup of $C=P^\vee/Q^\vee$,
$\tau\in\Aut\D$ is a diagram involution, and $\qq\in\Km(\Dtil)$ is a Kac labeling.

For any two Kac labelings $\pp,\qq\in\Km(\Dtil)$, set
$$\nu_{\pp,\qq}=2\bigl(x(\pp)-x(\qq)\bigr)/\ii=\sum_{\beta\in\Smbar}(p_\beta-q_\beta)\bar\omega^\vee_\beta\in P_0^\vee.$$
Let $\Km(\Dtil,X_0,\qq)$ denote the set of Kac labelings $\pp\in\Km(\Dtil)$ satisfying the condition
\begin{equation}\label{e:cond-q}
\langle\lambda,\pp\rangle\equiv\langle\lambda,\qq\rangle\pmod\Z\quad\text{for all }[\lambda]\in X_0/Q_0\hs.
\end{equation}
Recall that the finite abelian group $F_0=\wt X^\vee_0/\wt Q^\vee_0$ acts
by diagram automorphisms on $\Dtil$, and hence on $\Km(\Dtil)$.
\end{subsec}

\begin{main-theorem}\label{t:main-ss}
For $(D,F,\tau,\qq)$ as above, the group $F_0$, when acting on $\Km(\Dtil)$, preserves $\Km(\Dtil,X_0,\qq)$\hs.
We have $\nu_{\pp,\qq}\in X_0^\vee=\Hom(\GmC,T_0)$ whenever $\pp\in \Km(\Dtil,X_0,\qq)$, and the  map
\begin{equation}\label{e:Kac-H1}
\Km(\Dtil,X_0,\qq)\hs \to\hs (T_0)_2\into \Zl1(\R,\GG(D,F,\tau,\qq)),\quad \pp\mapsto \nu_{\pp,\qq}(-1)\in (T_0)_2
\end{equation}
induces a bijection between the set of $F_0$-orbits in $\Km(\Dtil,X_0,\qq)$
and the first Galois cohomology set $\Ho1(\R,\GG(D,F,\tau,\qq))$.
\end{main-theorem}

\begin{proof}
Recall that we have $\tto=\Exp(x(\qq))$ and $z=\tto^2=(\tto\cdot\htau)^2$.
Condition \eqref{e:cond-z} for a Kac labeling $\pp\in\Km(\Dtil)$
is equivalent to condition \eqref{e:cond-q},
since $\qq$ itself satisfies \eqref{e:cond-z} by Theorem~\ref{t:p-z}.
For $\pp\in\Km(\Dtil,X_0,\qq)$, it follows from \eqref{e:cond-q} that
$$\lambda(\nu_{\pp,\qq})=2\lambda\bigl(x(\pp)-x(\qq)\bigr)/\ii= \langle\lambda,\pp\rangle-\langle\lambda,\qq\rangle\in\Z\quad\text{for all }\lambda\in X_0,$$
that is, $\nu_{\pp,\qq}\in X_0^\vee$. Again by Theorem~\ref{t:p-z}, the map
\[ \Km(\Dtil,X_0,\qq)\hs\longrightarrow\hs(T_0\cdot\htau)_2^z\hs,\quad \pp\hs\longmapsto\hs\Exphat(x(\pp))=\Exp(x(\pp))\cdot\htau \]
induces a bijection $\Km(\Dtil,X_0,\qq)\hs/F_0\,\isoto\, (T_0\cdot\htau)_2^z\hs/\Wohat$\hs.
By  Corollary \ref{c:times-tau} the map
\[(T_0\cdot\htau)_2^z \,\longrightarrow\, (T_0)_2 \into \Zl1(\R,\GG(D,F,\tau,\qq)),\quad
          t\cdot\htau\longmapsto t\hs \tto^{-1} = t\cdot\Exp(-x(\qq)) \]
induces a bijection $(T_0\cdot\htau)_2^z\hs/\Wohat \isoto \Ho1(\R,\GG(D,F,\tau,\qq)).$
The composite bijection sends the $F_0$-orbit of $\pp$ to the cohomology class of
\[\Exp(x(\pp))\cdot\Exp(-x(\qq))=\Exp\bigl(x(\pp)-x(\qq)\bigr)=\exp 2\pi\bigl(x(\pp)-x(\qq)\bigr)= \exp\pi\ii\hs\nu_{\pp,\qq}=\nu_{\pp,\qq}(-1),\]
where the last equality follows from \eqref{e:exp-2-pi-i-z}. This completes the proof.
\end{proof}

Let us see what Theorem~\ref{t:main-ss} gives us
in the two particular cases of simply connected and adjoint groups.
A semisimple $\R$-group $\GG(D,F,\tau,\qq)$ is simply connected if and only if $F=0$.
In this case, we have $F_0=0$, $X=P$, and $X_0=P_0$.

\begin{corollary}\label{c:main-sc}
There is a canonical bijection between 
$\Km(\Dtil,P_0,\qq)$
and $\Ho1(\R,\GG(D,0,\tau,\qq))$ given by the map \eqref{e:Kac-H1}.
\end{corollary}

This result in the special case when $\tau=\id$
was stated by  E.~B.~Vinberg in a 2008 email message to the first-named author.

In the adjoint case, we have $X=Q$, $F=C$, and hence $X_0=Q_0$, $F_0=C_0$.
It follows that condition \eqref{e:cond-q} is empty and $\Km(\Dtil,Q_0,\qq)=\Km(\Dtil)$.

\begin{corollary}\label{c:Kac}
There is a canonical bijection between 
$\Km(\Dtil)/C_0$ and $\Ho1(\R,\GG(D,C,\tau,\qq))$
given by the map \eqref{e:Kac-H1}.
\end{corollary}

The adjoint $\R$-group $\GG(D,C,\tau,\qq)=\thG^\ad$ is the identity component of
the algebraic group $\Aut G$ equipped with the real structure given by the conjugation action of $\sigma=\theta\circ\sigma_c$.
It is well known that the cohomology classes in $\Ho1(\R,\thG^\ad)$ are in a bijective correspondence
with the $G^\ad$-conjugacy classes of real structures $\sigma'=\theta'\circ\sigma_c$ on $G$
such that $\theta'_\brd=\tau$; see Serre \cite[Section III.1]{Serre}.
Thus Corollary~\ref{c:Kac} gives a classification of inner
forms of a semisimple group $\thG$ in terms of Kac labelings.
This result goes back to Kac \cite{Kac0}, \cite{Kac};
see also \cite[Sections 5.1.4--5.1.6]{OV} and \cite[Sections 4.1.3--4.1.5]{OV2}.

\section{Twisting and functoriality}
\label{s:twisting}

\begin{subsec}\label{ss:twist-def}
Recall that if $\GG=(G,\sigma)$ is a real algebraic group and $a\in \Zl1(\R,\GG^\ad)$,
then the $a$-twist of $\GG$ is $\0\GG{a}:=(G,a\circ \sigma)$.
Here $\GG^\ad$ is identified with
the group $\Inn G$ equipped with the real structure given by the conjugation action of $\sigma$.
The same notation $\0\GG{a}$ is used for the twist of $\GG$
 by a cocycle $a\in \Zl1(\R,\GG)$ acting on $G$ by~$\inn(a)$.
See Serre \cite[Section I.5.3]{Serre}.

We use the notation of Subsection \ref{ss:can-R-str}.
Having fixed $\D$, $F$, and $\tau$, we shall write $\GG_\qq=\GG(D,F,\tau,\qq)$ for brevity.

Let $\Exp^\ad:\tl\to T^{\ad}$ denote the scaled exponential map $x\mapsto\exp(2\pi x)$ in the adjoint group.
Identifying $G^\ad$ with $\Inn G$, we have $\Exp^\ad(x)=\inn(\Exp(x))$.
\end{subsec}

\begin{proposition}\label{p:twisting-ad}
Let $\qq'\in \Km(\Dtil)$ and consider the 1-cocycle
$$a=\Exp^\ad\bigl(x(\qq')-x(\qq)\bigr)
\in (T^\ad_0)_2\subset \Zl1(\R,\GG_\qq^\ad).$$
Then there is a canonical isomorphism  $\0{\GG_\qq}a\isoto\GG_{\qq'}$\hs.
\end{proposition}

\begin{proof}
Let $\sigma_\qq$ denote the real structure on $G$ corresponding to the real form
$\GG_\qq$; then $\sigma_\qq=\Exp^\ad\bigl(x(\qq)\bigr)\circ\ttau\circ\sigma_c$.
Similarly, $\sigma_{\qq'}=\Exp^\ad\bigl(x(\qq')\bigr)\circ\ttau\circ\sigma_c$.
We see that
\[a\circ\sigma_\qq=
      \Exp^\ad\bigl(x(\qq')-x(\qq)\bigr)\circ\Exp^\ad\bigl(x(\qq)\bigr)\circ\ttau\circ\sigma_c
      =\Exp^\ad\bigl(x(\qq')\bigr)\circ\ttau\circ\sigma_c
      =\sigma_{\qq'}\hs,\]
as required.
\end{proof}

Now let $\qq'\in \Km(\Dtil,X_0,\qq)$, that is, $\qq'\in\Km(\Dtil)$
and $\langle\lambda,\qq'\rangle-\langle\lambda,\qq\rangle\in\Z$ for all $[\lambda]\in X_0/Q_0$.
With the notation of Theorem~\ref{t:main-ss},
consider the 1-cocycle $$a=\nu_{\qq'\!,\qq}(-1)=\Exp\bigl(x(\qq')-x(\qq)\bigr)\in (T_0)_2\subset \Zl1(\R,\GG_\qq)$$
and the twisting bijection
$\tw_a\colon \Ho1(\R,\0{\GG_\qq}a)\to \Ho1(\R,\GG_\qq)$
of Serre \cite[I.5.3, Proposition 35  bis]{Serre}.
By Proposition \ref{p:twisting-ad} we may identify $\0{\GG_\qq}a$ with $\GG_{\qq'}$.
Thus we obtain a map
\begin{equation*}
\tw_a\colon \Ho1(\R,\GG_{\qq'})\to \Ho1(\R,\GG_\qq)
\end{equation*}
sending the neutral cohomology class $[1]\in\Ho1(\R,\GG_{\qq'})$  to $[a]\in \Ho1(\R,\GG_\qq)$.

\begin{proposition}\label{p:inner-twist}
$\tw_a[\nu_{\pp,\qq'}(-1)]=[\nu_{\pp,\qq}(-1)]$ for all $\pp\in \Km(\Dtil,X_0,\qq')=\Km(\Dtil,X_0,\qq)$.
\end{proposition}

\begin{proof}
Note that $\nu_{\pp,\qq}=\nu_{\pp,\qq'}+\nu_{\qq',\qq}$.
The map $\tw_a$ is induced by the map on cocycles
$$\Zl1(\R,\GG_{\qq'})\to \Zl1(\R,\GG_\qq),\quad a'\mapsto a'a,$$
sending $\nu_{\pp,\qq'}(-1)$ to
\begin{equation*}
\nu_{\pp,\qq'}(-1)\cdot a= \nu_{\pp,\qq'}(-1)\cdot\nu_{\qq',\qq}(-1)= \nu_{\pp,\qq}(-1).\qedhere
\end{equation*}
\end{proof}

\begin{subsec}
Let
\[\varphi\colon \GG\to \GG'\]
be a {\em normal} homomorphism of not necessarily compact semisimple $\R$-groups.
Here ``normal" means that $\varphi(\GG)$ is normal in $\GG'$.
It is easy to describe the induced map
\begin{equation}\label{e:phi_*}
\varphi_*\colon \Ho1(\R,\GG)\to \Ho1(\R,\GG')
\end{equation}
using the descriptions of $\Ho1(\R,\GG)$ and $\Ho1(\R,\GG')$ given in Main Theorem \ref{t:main-ss}.
In Proposition \ref{p:phi_*} below we state the corresponding assertion
in the special case when $\varphi$ is an isogeny (a surjective homomorphism with finite kernel).
We leave the task of stating the general assertion to the interested reader.
\end{subsec}

\begin{subsec}
Let  $\varphi\colon \GG\to \GG'$ be an isogeny
of semisimple {\em not necessarily compact} $\R$-groups.
We compute the induced  map \eqref{e:phi_*} in Galois cohomology.

Let $\TT\subset \GG$ be a maximal torus containing a maximal compact torus $\TT_0$ of $\GG$.
Set $\TT'=\varphi(\TT)$, $\TT_0'=\varphi(\TT_0)\subseteq \TT'$.
Then $\TT'$ is a maximal torus of $\GG'$ and $\TT_0'$ is a maximal compact torus of $\GG'$.
Let $X,X',X_0,X'_0$ denote the character lattices of $T,T',T_0,T'_0$, respectively.
Write
\[\GG=\GG(D,F,\tau,\qq),\qquad \GG'=\GG(D,F',\tau,\qq)\]
(with the same $D,\tau,\qq$).
Here $F=X/Q$ and $F'=X'/Q$.
For any $\pp\in\Km(\Dtil)$, the vector $\nu_{\pp,\qq}$ may be regarded
as a cocharacter of $T_0$ if $\pp\in\Km(\Dtil,X_0,\qq)$,
and as a cocharacter of $T'_0$ if $\pp\in\Km(\Dtil,X_0'\hs,\qq)$,
in which case we denote it as $\nu^{\hssh\prime}_{\pp,\qq}$.
Then by Main Theorem \ref{t:main-ss} we have
\begin{align*}
\Ho1(\R,\GG)  &= \{\hs               [\nu_{\pp,\qq}(-1)] \hs\mid\hs \pp\in \Km(\Dtil,X_0,\qq)     \hs\}, \\
\Ho1(\R,\GG') &= \{\hs [\nu^{\hssh\prime}_{\pp,\qq}(-1)] \hs\mid\hs \pp\in \Km(\Dtil,X_0'\hs,\qq) \hs\}.
\end{align*}
Note that $\Km(\Dtil,X_0\hs,\qq)\subseteq\Km(\Dtil,X_0',\qq)$, since $X_0\supseteq X_0'$.
\end{subsec}

\begin{proposition}\label{p:phi_*}
For $\pp\in \Km(\Dtil,X_0\hs,\qq)$ we have  $\varphi(\nu_{\pp,\qq}(-1))=\nu^{\hssh\prime}_{\pp,\qq}(-1)$, and hence
\[\varphi_*[\nu_{\pp,\qq}(-1)]=[\nu^{\hssh\prime}_{\pp,\qq}(-1)].\]
\end{proposition}

\begin{proof} Obvious. \end{proof}

\section{Example: real forms of $\EEE_7$}
\label{s:E7}

Let $\GG$ be the simply connected compact group $\GG$ of type $\EEE_7$.
In the figure below, on the left we give the extended Dynkin diagram $\Dtil$ of $G$
with the numbering of vertices of Onishchik and Vinberg \cite[Table 7, Type I]{OV}.
We paint the lowest root in black.
On the right we give $\Dtil$  with the coefficients $m_i$
from \cite[Table 6]{OV}; see~\eqref{e:sum-mj}.

We have $X=P$, and $P/Q\simeq \Z/2\Z$ is generated by $[\lambda]$ for
\begin{equation}\label{e:lambda-E7}
\lambda=\half(\alpha_1+\alpha_3+\alpha_7),
\end{equation}
see \cite[Table 3]{OV}.
In the  diagram on the right we paint in black the roots appearing
(with  non-integer  half-integer coefficients) in formula \eqref{e:lambda-E7}:

\tikzset{/Dynkin diagram,mark=o,affine mark=*,
edge length=0.75cm,arrow shape/.style={-{angle 90}}}

\begin{align*}
&\dynkin[%
edge length=0.75cm,
backwards,
upside down,
affine mark=*,
labels={0,6,7,5,4,3,2,1},
label directions={,,left,,,,,}%
]%
E[1]{7}
\qquad\qquad
&\dynkin[%
edge length=0.75cm,
backwards,
upside down,
affine mark=o,
labels={1,2,2,,,,2,1},
label directions={,,left,,,,,}%
]
E[1]{o*oo*o*}
\end{align*}

The Kac labelings of $\Dtil$ are:
\begin{align*}
\qq^{(1)}&=000\two{0}{0}002\qquad \qq^{(2)}=200\two{0}{0}000 &
\qq^{(3)}&=100\two{0}{0}001 \\
\qq^{(4)}&=000\two{0}{0}010\qquad \qq^{(5)}=010\two{0}{0}000 &
\qq^{(6)}&=000\two{0}{1}000\,.
\end{align*}

Since all automorphisms of $G$ are inner,
the real forms of $\EEE_7$ correspond to the elements of $\Ho1(\R, \GG^\ad)$,
and by Corollary \ref{c:Kac} they correspond
to the orbits of $C=P^\vee/Q^\vee\simeq\Z/2\Z$ in the set  $\Km(\Dtil)$ of Kac labelings of $\Dtil$.
These orbits  are:
$$\{\qq^{(1)}, \qq^{(2)}\},\quad \{\qq^{(3)}\},\quad \{\qq^{(4)}, \qq^{(5)}\},\quad \{\qq^{(6)}\},$$
hence $\# \Ho1(\R,\GG^\ad)=4$.
We write $\GG_\qq$ for the real form of $\GG$ defined by the Kac labeling $\qq$.
In particular,  the compact form corresponds to $\qq^{(1)}$.

Concerning $\Ho1(\R,\GG_\qq)$, condition \eqref{e:cond-q} defining $\Km^{(\qq)}:=\Km(\Dtil,P,\qq)$ reads as
\begin{equation*}
\half(p_1+p_7) \equiv \half(q_1+q_7) \pmod{\Z},
\end{equation*}
(note that $p_3=0$ and $q_3=0$),
which is equivalent to
\begin{equation*}
p_1+p_7 \equiv q_1+q_7 \pmod{2}.
\end{equation*}
We say that a labeling $\pp\in\Km$ is {\em even} (resp.\ {\em odd})
if the sum over the black vertices
\[p_1+p_7\]
is even (resp.\ odd).
Then $\Km^{(\qq)}$ is the set of the labelings $\pp\in \Km$ of the same parity as $\qq$.
By Corollary~\ref{c:main-sc},
the first Galois cohomology set $\Ho1(\R, \GG_\qq)$
is in a bijection with the set $\Km^{(\qq)}$.

For $\GG_\qq=\EEE_{7(-133)}$ (the compact form) we take $\qq=\qq^{(1)}$.
For $\GG_\qq={\bf EVI}=\EEE_{7(-5)}$ we take $\qq=\qq^{(5)}$; see \cite[Table 7]{OV}.
Both labelings $\qq^{(1)}$ and $\qq^{(5)}$ are even.
We see that in both cases the set  $\Km^{(\qq)}$ is the set
$$\Km^\even=\{\qq^{(1)},\qq^{(2)},\qq^{(4)},\qq^{(5)}\}$$
of all {\em even} labelings of $\Dtil$.
The set $\Ho1(\R,\GG_\qq)$ is in a bijection with the set $\Km^\even$.
In particular, $\#\Ho1(\R,\GG_\qq)=4$ in both the compact case and ${\bf EVI}$.

For $\GG_\qq={\bf EV}=\EEE_{7(7)}$ (the split form) we take $\qq=\qq^{(6)}$,
and for $\GG_\qq={\bf EVII}=\EEE_{7(-25)}$ (the Hermitian form)
we take $\qq=\qq^{(3)}$; see \cite[Table 7]{OV}.
Both labelings $\qq^{(6)}$ and $\qq^{(3)}$ are odd.
In both cases the set  $\Km^{(\qq)}$ is the set
$$\Km^\odd=\{\qq^{(3)},\qq^{(6)}\}$$
of all {\em odd} labelings of $\Dtil$.
The set $\Ho1(\R, \GG_\qq)$ is in a bijection with the set $\Km^\odd$.
In particular,  $\#\Ho1(\R, \GG_\qq)=2$ in both cases ${\bf EV}$ and ${\bf EVII}$.

In each case the element $\qq\in\Km^{(\qq)}$ corresponds to the neutral element of $\Ho1(\R,\GG_\qq)$.

Note that  $\Ho1(\R,\EEE_{7(-5)})$ was earlier computed
by Garibaldi and Semenov \cite[Example 5.1]{GS},
and  $\Ho1(\R,\EEE_{7(7)})$ was earlier computed by Conrad \cite[Proof of Lemma 4.9]{Conrad}.
The cardinalities of  $\Ho1(\R,\GG)$ for all simply connected absolutely simple groups $\GG$,
in particular, for all four real forms of the simply connected group of type $\EEE_7$,
were computed by Adams and Ta\"\i bi \cite{AT} and also by Borovoi and Evenor \cite{BE}.

\section{Example: half-spin groups}
\label{s:half-spin}

Let $\GG$ be the compact group of type $\DDD_\ell$
 with even $\ell\ge 4$ with the cocharacter lattice
\[ X^\vee =\langle Q^\vee,\omega_{\ell-1}^\vee\rangle. \]
This compact group is neither simply connected nor adjoint,
and it is  isomorphic to $\SO_{2\ell}$ only if  $\ell=4$.
It is called a half-spin group.

In the figure below, on the left we give the extended Dynkin diagram $\Dtil$ of $G$
with the numbering of vertices of
Onishchik and Vinberg \cite[Table 7, Type I]{OV}. We paint the lowest root in black.
On the right we give $\Dtil$
with the coefficients $m_i$ from \cite[Table 6]{OV}; see \eqref{e:sum-mj}.

We show that the character lattice $X$  is generated by $Q$ and the weight
\begin{equation}\label{e:X-alphas}
 \lambda:=(\alpha_1+\alpha_3+\alpha_5+\dots+\alpha_{\ell-3}+\alpha_\ell)/2.
\end{equation}
Indeed, $\langle\lambda,\omega^\vee_{\ell-1}\rangle=0$
and $\langle\lambda,\alpha^\vee\rangle=0,1,-1\in\Z$ for any $\alpha\in\Sm$.
Since $\omega^\vee_{\ell-1}$ and the coroots  $\alpha^\vee$ generate $X^\vee$, we
see that $\lambda\in X$.
Since $\lambda\notin Q$ and $[X:Q]=2$, we conclude that $X=\langle Q,\lambda\rangle$.
In the diagram on the right we paint in black the roots
that appear (with non-integer half-integer coefficients)
in formula \eqref{e:X-alphas} for $\lambda$.

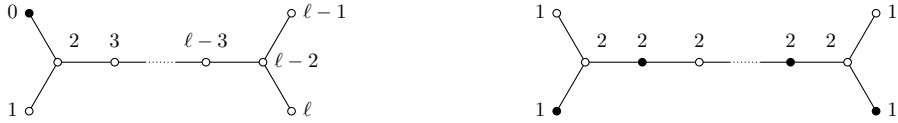
\tikzset{/Dynkin diagram,mark=o,
edge length=0.75cm,arrow shape/.style={-{angle 90}}}

\begin{align*}
\dynkin[%
labels*={,,2,3,\ell-3,,,},
labels={0,1,,,,\ell-2,\ell-1,\ell},
label directions={,,,,,right,,}
]
D[1]{}
\qquad\qquad\qquad
\dynkin[edge length=0.75cm,
affine mark=o,
arrow shape/.style={-{angle 90}},
labels={1,1,,,,,,1,1},
labels*={,,2,2,2,2,2,,}
]
D[1]{*o*o.*oo*}
\end{align*}

Let $\pp$ be a Kac labeling of $\Dtil$.
We say that $\pp$ is {\em even} (resp., {\em odd}) if the sum over the black vertices
\[p_1+p_3+p_5+\dots+p_{\ell-3}+p_\ell\]
is even (resp., odd).
If $\qq\in\Km(\Dtil)$, then by \eqref{e:cond-q} the set
$\Km^{(\qq)}:=\Km(\Dtil,X,\qq)$ consists of all
Kac labelings of the same parity as $\qq$.

The group $F=X^\vee/Q^\vee=\{\hs 0,\hs[\omega_{\ell-1}^\vee]\hs\}$ acts on $\Dtil$
and on $\Km(\Dtil)$.
The nontrivial element $[\omega_{\ell-1}^\vee]\in F$
acts as the reflection with respect to the vertical  axis of symmetry of $\Dtil$
(see Subsection \ref{ss:C-action}) and clearly preserves the parity of labelings.
We say that an $F$-orbit in $\Km$ is even (resp., odd),
if it consists of even (resp., odd) labelings.

Let $\qq$ be a Kac labeling of $\Dtil$.
By Main Theorem \ref{t:main-ss},
the cohomology set $\Ho1(\R,\GG_\qq)$ is in a  bijection with the set $\Km^{(\qq)}/F$,
that is, with the set of $F$-orbits in $\Km$ of the same parity as $\qq$.
Thus in order to compute  $\Ho1(\R,\GG_\qq)$ for all labelings $\qq$ of $\Dtil$, it suffices to compute
the sets of the even and odd $F$-orbits, respectively.
We compute also the cardinalities
$h^\even(\DDD_\ell)$ and $h^\odd(\DDD_\ell)$ of these sets.

For representatives of even $F$-orbits we take
\begin{align*}
&\gf{1}{0}0 \cdots 0\gf{1}{0}, \qquad
\gf{0}{1}0 \cdots 0\gf{0}{1}, \qquad
\gf{2}{0}0 \cdots 0\gf{0}{0}, \qquad
\gf{0}{2}0 \cdots 0\gf{0}{0}, \qquad
\text{and} \\& \\
&\gf{0}{0}0 \cdots 1 \cdots 0\gf{0}{0} \quad
\text{with $1$ at $2j$ for each integer $j$ with $1<2j\le \ell/2$.}
\end{align*}
Thus
\[h^\even(\DDD_{\ell})=\lfloor \ell/4\rfloor+4.\]

For representatives of odd $F$-orbits we take
\[
\gf{1}{1}0 \cdots 0 \gf{0}{0}, \qquad\quad
\gf{1}{0}0 \cdots 0 \gf{0}{1}, \qquad\quad
\text{and} \quad
\gf{0}{0}0 \cdots 1 \cdots 0\gf{0}{0} \quad
\text{with $1$ at $2j+1$}
\]
for each integer $j$ with $1<2j+1\le \ell/2$.
Thus
\[ h^\odd(\DDD_{\ell})=\lceil \ell/4\rceil+1.\]

We conclude that if $\qq$ is an even labeling, then $\#\Ho1(\R,\GG_\qq)=\lfloor \ell/4\rfloor+4$,
while  if $\qq$ is an odd labeling, then $\#\Ho1(\R,\GG_\qq)=\lceil \ell/4\rceil+1$.

Note that if $\ell>4$, then our compact half-spin group $\GG$ has no outer automorphisms;
hence all its real forms are {\em inner}  forms,
and we have computed the Galois cohomology for all real forms of $\GG$.

Note also that for the compact half-spin group $\GG$ we have
\[ \# \Ho1(\R,\GG)=h^\even(\DDD_{\ell})=\lfloor \ell/4\rfloor+4. \]
For comparison, $\#\Ho1(\R,\SO_{2\ell})=\ell+1$.
We have $\lfloor \ell/4\rfloor+4=\ell+1$ for an even natural number $\ell$ if and only if $\ell=4$.
(In this case, because of triality, our half-spin group $\GG$ is isomorphic to $\SO_8$.)

\section{Example: an almost direct product of $\EEE_7$ and $\SL_{m,{\Hquat}}$.}
\label{s:E7xSL(m,H)}

Let $\GG^\ssc=\EEE_7\times\SL_{m,{\Hquat}}$, where by $\EEE_7=\EEE_{7(-133)}$
we denote the compact simply connected $\R$-group of type $\EEE_7$,
and $\SL_{m,{\Hquat}}$ is the quaternionic real form of $\SL_{2m}$.
Let $\mu_{2m}$ denote the cyclic group generated by a primitive $2m$-th root of unity~$\zeta_{2m}$.
We identify $\mu_{2m}$ with the center of $\SL_{2m}(\C)$
consisting of scalar matrices, and embed $\mu_{2m}$ into $Z(G^\ssc)$
by sending $\zeta_{2m}$ to $(-1,\zeta_{2m})$, where $-1$
denotes the nontrivial central element of $\EEE_7$.
We set $\GG=\GG^\ssc/\mu_{2m}$\hs. Then $\GG$ is a double cover
of the adjoint group $\EEE_7^\ad\times\PGL_{m,{\Hquat}}$.
We wish to compute the Galois cohomology for all inner forms of $\GG$
(which are outer forms of a compact form of $\GG$).

By Corollary \ref{c:Kac}, the inner forms of $\GG$ are classified
by Kac labelings $\qq$ of the diagram $\Dtil=\Dtil_\eee\sqcup\Dtil_\aaa$,
where $\Dtil_\eee$ is as in Section \ref{s:E7}
and $\Dtil_\aaa$ is given at the figure below.
We denote various objects (restricted roots, Kac labelings, \dots) related to
$\Dtil_\eee$ or $\Dtil_\aaa$ by the subscripts (or superscripts)
$\eee$ and $\aaa$, respectively.

On the left-hand side of the figure below we give the twisted affine Dynkin diagram $\Dtil_\aaa$
of $\SL_{m,{\Hquat}}$ with the numbering of vertices of \cite[Table 7, Type III]{OV}.
We paint the root $\alpha^\aaa_0$ in black.
On the right-hand side we give $\Dtil_\aaa$ with the coefficients $m_i$
from the Table in \cite[Section 3.3.9]{OV2}; see \eqref{e:mj-outer}.

The group $X_0/Q_0\simeq \Z/2\Z$ is generated by $[\lambda]$ with
\begin{equation}\label{e:lambda-E7-A}
\lambda=\half(\alpha^\eee_1+\alpha^\eee_3+\alpha^\eee_7+\alpha^\aaa_m).
\end{equation}
In the diagram on the right we paint in black the root $\alpha^\aaa_m$ appearing
(with  non-integer  half-integer coefficient) in formula \eqref{e:lambda-E7-A}:

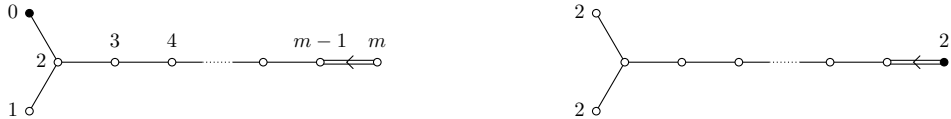
\tikzset{/Dynkin diagram,mark=o,affine mark=*,
edge length=0.75cm,arrow shape/.style={-{angle 90}}}

\begin{align*}
\dynkin[%
edge length=0.75cm,
arrow shape/.style={-{angle 90}},
labels={0,1,2},
labels*={,,,3,4,,m-1,m}
]
A[2]{odd}
\qquad\qquad\qquad
\dynkin[%
affine mark=o,
labels={2,2,,,,,,},
labels*={,,,,,,,2},
odd
]
A[2]{oooo.oo*}
\end{align*}

The Kac labelings $\qq_\eee^{(1)},\dots,\qq_\eee^{(6)}$ of $\Dtil_\eee$
are written down in Section~\ref{s:E7}.
The labelings $\qq_\eee^{(1)}$, $\qq_\eee^{(2)}$, $\qq_\eee^{(4)}$, and $\qq_\eee^{(5)}$ are even,
while the labelings $\qq_\eee^{(3)}$ and $\qq_\eee^{(6)}$ are odd.

We write down the Kac labelings of $\Dtil_\aaa$:
\[
\qq_\aaa^{(1)}=   \gf{1}{0}0 \cdots  00, \qquad \quad
\qq_\aaa^{(2)}=   \gf{0}{1}0 \cdots 00, \qquad \quad
\qq_\aaa^{(3)}=   \gf{0}{0}0 \cdots 01.
\]
The labelings $\qq_\aaa^{(1)}$ and $\qq_\aaa^{(2)}$ are even
(that is, the labels $\qq_{\aaa,m}^{(1)}$ and $\qq_{\aaa,m}^{(2)}$ are even),
while the labeling $\qq_\aaa^{(3)}$ is odd.

We have $\#\Km(\Dtil_\eee)=6$, $\#\Km(\Dtil_\aaa)=3$;
hence $\#\Km(\Dtil)=6\cdot 3=18$.

We say that a labeling $\qq=(\qq_\eee,\qq_\aaa)$
is \emph{even} (resp.\ \emph{odd}) if the sum
\[q^\eee_1+q^\eee_3+q^\eee_7+q^\aaa_m\]
is even (resp.\ odd).
Clearly, $\qq=(\qq_\eee,\qq_\aaa)$ is even if and only if
either both labelings $\qq_\eee$ and $\qq_\aaa$ are even or they both are odd.
By \eqref{e:cond-q} the set $\Km(\Dtil,X_0,\qq)$
consists of all Kac labelings of the same parity as $\qq$.

The group $F_0=\wt X_0^\vee/\wt Q_0^\vee$ is of order 2 with the nontrivial element
acting on the diagrams $\Dtil_\eee$ and $\Dtil_\aaa$
by the only nontrivial automorphisms of these diagrams,
respectively. This action clearly preserves the parity  of $\qq=(\qq_\eee,\qq_\aaa)$.
We say that an orbit of $F_0$ in $\Km(\Dtil)$ is even (resp.\ odd)
if it consists of even (resp.\ odd) labelings.

We write down all even $F_0$-orbits in $\Km(\Dtil)$:
\begin{align*}
&(\qq_\eee^{(1)},\qq_\aaa^{(1)})\longleftrightarrow (\qq_\eee^{(2)},\qq_\aaa^{(2)}),
    \qquad &(\qq_\eee^{(1)},\qq_\aaa^{(2)})\longleftrightarrow (\qq_\eee^{(2)},\qq_\aaa^{(1)}),\\
&(\qq_\eee^{(4)},\qq_\aaa^{(1)})\longleftrightarrow (\qq_\eee^{(5)},\qq_\aaa^{(2)}),
   \qquad &(\qq_\eee^{(4)},\qq_\aaa^{(2)})\longleftrightarrow (\qq_\eee^{(5)},\qq_\aaa^{(1)}),\\
&(\qq_\eee^{(3)},\qq_\aaa^{(3)}),  &(\qq_\eee^{(6)},\qq_\aaa^{(3)}).
\end{align*}

We write down all  odd $F_0$-orbits in $\Km(\Dtil)$:
\begin{align*}
&(\qq_\eee^{(1)},\qq_\aaa^{(3)})\longleftrightarrow (\qq_\eee^{(2)},\qq_\aaa^{(3)}),
     \qquad &(\qq_\eee^{(4)},\qq_\aaa^{(3)})\longleftrightarrow (\qq_\eee^{(5)},\qq_\aaa^{(3)}),\\
&(\qq_\eee^{(3)},\qq_\aaa^{(1)})\longleftrightarrow (\qq_\eee^{(3)},\qq_\aaa^{(2)}),\qquad
&(\qq_\eee^{(6)},\qq_\aaa^{(1)})\longleftrightarrow (\qq_\eee^{(6)},\qq_\aaa^{(2)}).
\end{align*}

We see that there are six even $F_0$-orbits and four odd $F_0$-orbits in $\Km(\Dtil)$.

Now let $\GG_\qq$ denote the twisted form of our group $\GG=(\EEE_7\times\SL_{m,{\Hquat}})/\mu_{2m}$
corresponding to the Kac labeling $\qq=(\qq_\eee,\qq_\aaa)$.
By Main Theorem \ref{t:main-ss}, if $\qq$ is even, then $\Ho1(\R,\GG_\qq)$
is in a bijection with the set of even $F_0$-orbits in $\Km(\Dtil)$ and hence $\#\Ho1(\R,\GG_\qq)=6$.
If $\qq$ is odd, then $\Ho1(\R,\GG_\qq)$ is in a bijection
with the set of odd $F_0$-orbits in $\Km(\Dtil)$ and hence $\#\Ho1(\R,\GG_\qq)=4$.

\end{document}